\documentclass{amsart}
\usepackage{amsthm}
\usepackage{amssymb}
\usepackage{amsmath}
\usepackage{multirow}
\usepackage{booktabs}
\usepackage{array}
\usepackage{graphicx}
\usepackage{graphics}
\usepackage{subfig}
\usepackage{threeparttable}
\usepackage{multicol}
\usepackage{lipsum}
\usepackage[monochrome]{color}
\usepackage{cuted}
\usepackage{geometry}
\usepackage{caption}

\usepackage[colorlinks, linkcolor=green, urlcolor=black]{hyperref}

\newcommand\restr[2]{\ensuremath{\left.#1\right|_{#2}}}

\newtheorem{thm}{Theorem}

\newtheorem{df}[thm]{Definition}
\newtheorem{lema}[thm]{Lemma}
\newtheorem{dosledok}[thm]{Corollary}

\newcommand{\tDET}[2]{\mathrm{DET}_{#1}^{#2}}
\newcommand{\tdet}[2]{\mathrm{det}_{#1}^{#2}}
\newcommand\utdet{\underline{\mathrm{det}}}
\newcommand\otdet{\overline{\mathrm{det}}}
\newcommand{\trr}[2]{\mathrm{rr}_{#1}^{#2}}
\newcommand{\tRR}[2]{\mathrm{RR}_{#1}^{#2}}
\newcommand{\tc}[2]{\mathrm{c}_{#1}^{#2}}
\newcommand{\tC}[2]{\mathrm{C}_{#1}^{#2}}

\newcommand\ve{r}
\newcommand{\dis}{\varrho}
\newcommand{\ZZZ}{\mathbb{Z}}
\newcommand{\NNN}{\mathbb{N}}
\newcommand{\RRR}{\mathbb{R}}
\newcommand{\abs}[1]{\lvert #1 \rvert}

\newcommand{\diam}{\operatorname{diam}}

\newcommand{\lp}{+} 
\newcommand{\pp}{\oplus} 
\newcommand{\lm}{-}
\newcommand{\pminus}{\ominus}
\newcommand{\eper}{\textnormal{ePer}}
\newcommand{\cores}{\kappa}
\newcommand{\card}{\operatorname{card}}
\newcommand{\Orb}{\mathrm{Orb}}
\newcommand\id{\mathrm{id}}

\begin{document}

\title[Correlation integral and determinism for a family of $2^\infty$ maps]{Correlation integral and determinism \\for a family of $2^\infty$ maps}

\author{J. Majerov\'a}

\address{
Department of Mathematics, Faculty of Natural Sciences, Matej Bel University, Tajovsk\'eho 40, 97401 Bansk\'a Bystrica, Slovakia\newline \indent
Slovanet a.s., Z\'ahradn\'icka 151, 821 08 Bratislava, Slovakia}
\email{majerova.jana@yahoo.com}


\begin{abstract}
The correlation integral and determinism are quantitative characteristics of a dynamical system based on the recurrence of orbits.
For strongly non-chaotic interval maps, the determinism equals $1$ for every small enough threshold. This means that trajectories of such systems are perfectly predictable in the infinite horizon.
In this paper we study the correlation integral and determinism for the family of $2^\infty$ non-chaotic maps, first considered by Delahaye in 1980. The determinism in a finite horizon equals $1$. However, the behaviour of the determinism in the infinite horizon is counter-intuitive. Sharp bounds on the determinism are provided.
\end{abstract}

\keywords{adding machine; determinism; correlation integral}

\date{\today}

\maketitle


\section{Introduction}
\label{del:intro}

The correlation integral was first introduced in~\cite{GP-cor-sum} to measure a quantity of recurrences of trajectory. It is tightly connected with the correlation dimension and correlation entropy. For example, these characteristics are used in chaos theory and in time series analysis, see e.g.~\cite{KS-tsa}.

Let $g$ be a continuous map on a compact metric space $(M, \dis)$. Throughout the paper we consider the distance $\sup\{\dis(g^i(x), g^i(y))\ |\ 0\leq i < \ell\}$, denoted by $\dis_{\ell}(x, y)$, where $\ell\in \NNN\cup\{\infty\}$. For a fixed ergodic measure $\mu$ on $M$, the $\ell$--\emph{correlation integral} $\tc{\ell}{g}(\mu, \ve)$ for $\ve > 0$ is a measure of pairs of points which are $\ve$--close with respect to the metric $\dis_{\ell}$. Then $\tc{\ell_2|\ell_1}{g}(\mu, \ve) = \tc{\ell_2 + \ell_1}{g}(\mu, \ve)/\tc{\ell_1}{g}(\mu, \ve)$, where $1\leq\ell_1 < \infty$ and $1\leq\ell_2\leq\infty$, is a conditional probability. Moreover, it is a measure of partial predictability. Especially, $\tc{\infty|1}{g}(\mu, \ve)$ is a measure of total predictability. The natural question arises whether the predictability of a system can be maximal for all small enough $\ve$'s. This holds for $g$ being a contraction, an isometry, or whenever the support of $\mu$ is a periodic orbit. In contrary, if a map $g$ has sensitive dependence on initial conditions and $\mu$ is not atomic then $\tc{\infty|1}{g}(\mu, \ve)$ can be equal to zero for all small enough thresholds $\ve > 0$. Hence, we can say that this measure quantifies how sensitive a map is towards initial conditions on the support of $\mu$. We call it an \emph{$\infty$--asymptotic determinism} or an asymptotic determinism in the infinite horizon and denote it by $\tdet{\infty}{g}(\mu, \ve)$. Similarly, an $\ell$--\emph{asymptotic determinism} is defined as a linear combination of the conditional probabilities
\begin{align*}
&\tdet{\ell}{g}(\mu, \ve) = \ell\cdot \frac{\tc{\ell}{g}(\mu, \ve)}{\tc{1}{g}(\mu, \ve)} - (\ell - 1)\cdot \frac{\tc{\ell + 1}{g}(\mu, \ve)}{\tc{1}{g}(\mu, \ve)} = \ell\cdot \tc{\ell - 1|1}{g}(\mu, \ve) - (\ell - 1)\cdot \tc{\ell|1}{g}(\mu, \ve).
\end{align*}

The map $g_{1/3}$ goes back to the work of Delahaye, cf.~\cite{delahaye}. There has since been systematic work on this map together with its variations (e.g.~\cite{CK-delahaye}, \cite[p.~137]{Dev-book}, \cite{H-delahaye}). It is often presented as the simplest non-chaotic $2^\infty$ map. In this paper we study the $\ell$--correlation integral and the $\ell$--asymptotic determinism for a special family of interval maps denoted by $\{f_{\alpha}\}_{0 < \alpha < 1/2}$. Especially, $f_{1/3} = 1 - g_{1/3}\circ(1 - \id_{[0, 1]})$, e.g.~defined in \cite[Example~5.54]{Rue-chaos}. These maps are non-chaotic in the sense of Li and Yorke and conjugate to an isometry on the support of the unique non-atomic ergodic measure $\mu_{\alpha}$. Some properties of these imply the maximal predictability for every ergodic measure and for small enough $\ve$. Theorem~\ref{thm:determinizmus_ohranicenie_tretina} shows that the hypothesis is wrong pointing out the fact that for $\alpha > 1/3$ and the non-atomic ergodic measure $\mu_{\alpha}$, the predictability is never maximal. Moreover, for each $\underline{det}\in(1/3, 8/15]$, we can find $\alpha\in(0, 1/2)$ satisfying $\lim\inf_{\ve\to 0^+}\tdet{\infty}{f_{\alpha}}(\mu_{\alpha}, \ve) = \underline{det}$ (Corollary~\ref{cor:liminf}). However, by Theorem~\ref{thm:det_jeden}, measures of partial predictability $\tc{\ell|1}{g}(\mu_{\alpha}, \ve) = 1$ for all $\ve \leq \ve_{\ell}$.

By Lemma~\ref{lema:det_epsilon}, it is not necessary to study the asymptotic determinism in the infinite horizon for $\mu_{\alpha}$ on the interval $[0, 1]$. Since $\tdet{\infty}{f_{\alpha}}(\mu_{\alpha}, \ve) = \tdet{\infty}{f_{\alpha}}(\mu_{\alpha}, \alpha\cdot \ve)$ for each $\ve \leq (1 - 2\alpha)/\alpha$, all properties can be determined from $\tdet{\infty}{f_{\alpha}}(\mu_{\alpha}, [\alpha^h(1 - 2\alpha), \alpha^{h - 1}(1 - 2\alpha)])$ where $h > 0$ is such that $\alpha^{h - 1}(1 - 2\alpha) \leq 1$. In addition, it is not necessary to use the map $f_{\alpha}$ for the computation either. By Corollary~\ref{cor:rozdiel_ck_c}, for approximation with a desired accuracy, we can use a simpler map $f_{\alpha, k}$ and $f_{\alpha, k}$--ergodic measure $\mu_{\alpha, k}$. All points from the support of $\mu_{\alpha}$ are periodic for $f_{\alpha, k}$ and it is sufficient to investigate only a finite number of pairs of points which can be written into the matrix which contains the patterns simplifying the computation (Lemma~\ref{lema:postupnost_gulky}).

For all $x\in [0, 1]$, there is $f_{\alpha}$--ergodic measure $\mu_{\alpha, x}$ such that for every $\ell \leq \infty$ the limit of $\card\{(i, j)\ |\ 0\leq i, j < n, \dis_{\ell}(f_{\alpha}^i(x), f_{\alpha}^j(x))\leq \ve\}/n^2$ is the correlation integral $\tc{\ell}{f_\alpha}(\mu_{\alpha, x}, \ve)$ (Theorem~\ref{thm:ostatok_body}). Moreover, if $x = 0$, then the measure $\mu_{\alpha, 0}$ is the unique non-atomic ergodic measure $\mu_{\alpha}$ (Lemma~\ref{lema:c_l_limita}). The functions $\tc{\ell}{f_{\alpha}}(\mu_{x}, \ve)$ and $\tdet{\ell}{f_\alpha}(\mu_{x}, \ve)$ are not continuous in the universal measure $\mu_{\alpha}$. If $x$ is not eventually periodic then the correlation integrals and asymptotic determinisms are continuous with respect to the radius $\ve$ (Theorem~\ref{thm:prechod}) and parameter $\alpha$ (Lemma~\ref{lema:int_conti}).

\section{Preliminaries}
\label{del:preliminaries}
\label{del-addm}

Let a dynamical system $(M, g)$ be given where $M$ is a metric space with metric $\dis$. Define the metric $\dis_{\ell}$ by
\begin{align*}
\dis_{\ell}(x, y) = \sup\{\dis(g^i(x), g^i(y))\ |\ 0\leq i < \ell\}
\end{align*}
where $1\leq \ell\leq \infty$. The \emph{$\ell$--correlation sum} for $\ve > 0$ was first defined in~\cite{GP-cor-sum} as
\begin{equation}
\label{eq:cor_sum}
\tC{\ell}{g}(x, n, \ve) = \frac{1}{n^2}\card\{(i, j)\ |\ 0\leq i, j < n, \dis_{\ell}(g^i(x), g^j(x))\leq\ve\}.
\end{equation}
Similarly, the \emph{$\ell$--recurrence rate} and \emph{$\ell$--determinism} (notions of the \emph{recurrence quantification analysis}, cf.~\cite{Z-rqa}) are defined by
\begin{equation}
\label{eq:rec_sum}
\begin{split}
&\tRR{\ell}{g}(x, n, \ve) = \frac{1}{n^2}\card\{(i, j)\ |\ 0\leq i, j < n, \eta_{\ell}(i, j, \ve) = 1\},\\
&\tDET{\ell}{g}(x, n, \ve) = \frac{\tRR{\ell}{g}(x, n, \ve)}{\tRR{1}{g}(x, n, \ve)}
\end{split}
\end{equation}
where $\eta_{\ell}(i, j, \ve) = 1$, if there is $0\leq k\leq \min\{i, j, \ell - 1\}$ such that $\dis_{\ell}(g^{i - k}(x), g^{j - k}(x))\leq \ve$, and $h(i, j, \ve) = 0$ otherwise. If $\ell = \infty$, we define $\min\{i, j, \infty - 1\} = \min\{i, j\}$ for all $0\leq i, j < \infty$.

The next lemma shows that for the computation of determinism and recurrence rate it is sufficient to know only the correlation sums. Its first part was proved in~\cite{strong-laws}, hence only the case $\ell = \infty$ remains to prove.

\begin{lema}
\label{lema:korelacny_rekurent}
Let $(M, g)$ be a dynamical system where $M$ is a metric space. For all $n > 0, 1\leq \ell < \infty$, $x\in M$ and $\ve > 0$ the recurrence rate $\tRR{\ell}{g}(x, n, \ve) = \ell\cdot \tC{\ell}{g}(x, n, \ve) - (\ell - 1)\cdot \tC{\ell + 1}{g}(x, n, \ve)$. If $\ell = \infty$, so $\tRR{\infty}{g}(x, n, \ve) = \tC{\infty}{g}(x, n, \ve)$.
\end{lema}
\begin{proof}
By definition, $\tC{\infty}{g}(x, n, \ve)\leq \tRR{\infty}{g}(x, n, \ve)$. For $\ell = \infty$, consider the pair $(i, j)$ with the property that $\dis_{\infty}(g^{i - k}(x), g^{j - k}(x))\leq \ve$ for some $0\leq k\leq \min\{i, j\}$. Therefore $\dis_{\infty}(g^{i}(x), g^{j}(x))\leq \ve$. It follows that $\tRR{\ell}{g}(x, n, \ve)\leq \tC{\ell}{g}(x, n, \ve)$.
\end{proof}

From Lemma~\ref{lema:korelacny_rekurent},
\begin{equation}
\label{eq:det_nekonecno}
\begin{aligned}
&\tDET{\ell}{g}(x, n, \ve) = \frac{\ell\cdot \tC{\ell}{g}(x, n, \ve) - (\ell - 1)\cdot \tC{\ell + 1}{g}(x, n, \ve)}{\tC{1}{g}(x, n, \ve)}\quad \text{for}\ 1\leq\ell < \infty,\\
&\tDET{\infty}{g}(x, n, \ve) = \frac{\tC{\infty}{g}(x, n, \ve)}{\tC{1}{g}(x, n, \ve)}.
\end{aligned}
\end{equation}

In this paper we consider the generalized rotated version of Delahaye's $g_{1/3}$ (cf.~\cite{delahaye}), namely $f_{\alpha}$. In special case $f_{1/3} = 1 - g_{1/3}\circ(1 - \id_{[0, 1]})$. Maps $f_{\alpha}$ are usually defined by sequences of maps. But for us it is more convenient to work with an exact definition.

\begin{df}
\label{df:del}
Let $\alpha\in(0, 1/2)$ and $k\geq 1$. Then $f_\alpha$ is defined by
\begin{eqnarray*}
f_{\alpha}(x) = \left\{
\begin{array}{ll}
x - 1 + 2\alpha^{j - 1} - \alpha^j                                                              & x\in[1 - \alpha^{j - 1}, 1 - \alpha^{j - 1} + \alpha^j], j\geq 1,\\
\frac{1 - \alpha + \alpha^2}{2\alpha - 1}(x - 1) + \alpha^{j + 1}\frac{2 - \alpha}{2\alpha - 1} & x\in(1 - \alpha^{j - 1} + \alpha^{j}, 1 - \alpha^{j}), j\geq 1,\\
0																								& x = 1.
\end{array}
\right.
\end{eqnarray*}
We call the approximation map $f_{\alpha, k}: [0, 1]\to [0, 1]$ a continuous map equal to $f_{\alpha}$ on the interval $[0, 1 - \alpha^{k - 1} + \alpha^k]$, it is equal to $x - 1 + \alpha^k$ on the interval $[1 - \alpha^k, 1]$, and it is linear elsewhere. Define the function $f_0: [0, 1]\to [0, 1]$ by $f_0(x) = x$.
\end{df}

\begin{figure}[!ht]
    \centering
        \includegraphics[width = 1\textwidth]{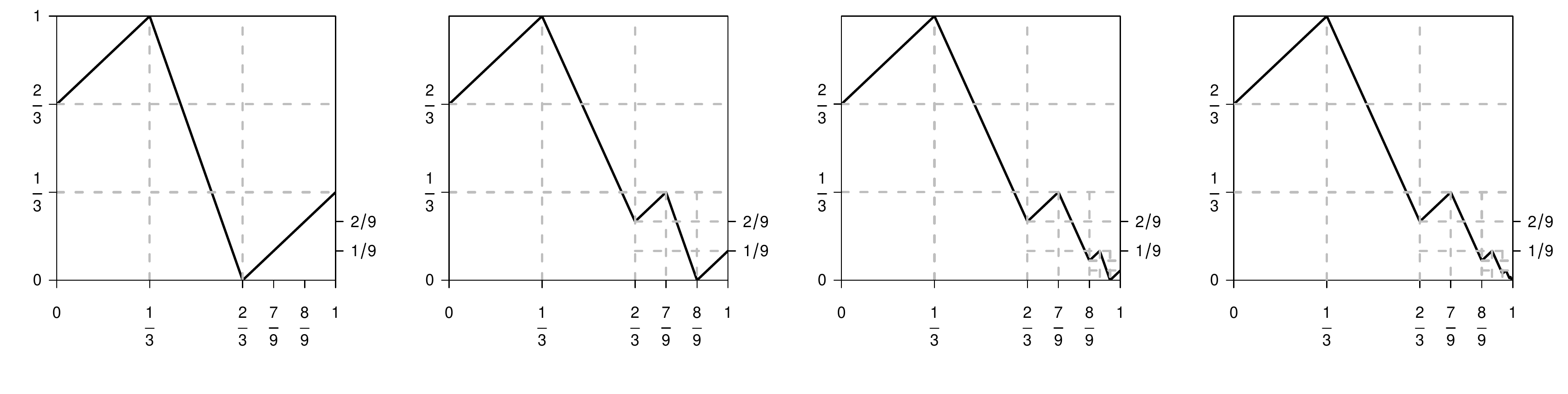}
    \caption[Examples of maps $f_\alpha, f_{\alpha, k}$]{From the left to the right $f_{1/3, 1}, f_{1/3, 2}, f_{1/3, 3}, f_{1/3}$.}
    \label{fig:delahaye}
\end{figure}

If $\alpha \in (0, 1/2)$ is known and there is no uncertainty about its value, then we write $f$ and $f_k$ instead of $f_{\alpha}$ and $f_{\alpha, k}$. For brevity, we omit $\alpha$ when there is no risk of confusion.

From now on, $\alpha$ and $\ve$ are always assumed to be fixed, $\alpha\in (0, 1/2)$ and $\ve\in(0, 1]$, unless stated otherwise (e.g.~we study the correlation integral as a function of these arguments).

\subsection{Adding Machine}

In this section, we provide some information about the adding machine and its connection with the dynamics of $f_{\alpha}$ and $f_{\alpha, k}$ restricted to their attractors. For a more abstract approach see e.g.~\cite{M-add-mach}.

Put $\Sigma = \{0,1\}$ and let $1\leq k \leq\infty$. To simplify the notation, we set $\infty + c = \infty$ for every $c\in\RRR$. The members $u$ of $\Sigma^k$ are called \emph{words} of length $\abs{u}=k$. For $u\in\Sigma^k$ and $1\leq i\leq j < k + 1$, denote $u_i^j=u_i u_{i+1}\ldots u_j$. If there is an ambiguity about the length of words, e.g~we use words of different lengths, we write $u^{(k)}, u_1^k\in\Sigma^k$. The \emph{concatenation} of words is understood as usually. By $0^k$, $1^k$ we denote the words $00\dots 0$, $11\dots 1$ from $\Sigma^k$. Symbols $v^{(0)}$, $v_1^0$, $1^0$, $0^0$ all denote the word of the length zero, the \emph{empty word}. It has the following properties: 
\begin{itemize}
	\item $u^{(0)} u^{(k)} = u^{(k)}$ for every word $u^{(k)}$ of the length $k\geq 1$ and
	\item $u^{(k)} u^{(0)} = u^{(k)}$ for every word $u^{(k)}$ of the length $k < \infty$. 
\end{itemize}	
We denote the set containing the empty word by $\Sigma^0$. For $u\in \Sigma^k$ and $v\in \Sigma^m$ where $v = u_1^m$ we write $u\succeq v$ and we say that $u$ \emph{begins with} $v$ if $m < k + 1$ and if $v = u_1^m$. 

For $1 \leq k \leq \infty$, denote the \emph{addition} from the left to the right on $\Sigma^k$ by $\lp$; e.g.~$100 \lp 110 = 001$.

We are led to the following lemma which is well known, e.g.~\cite{BC-1d}.

\begin{lema}
\label{lema:grupa}
For each $1\leq k\leq\infty$, $(\Sigma^k, \lp)$ is a cyclic group. If $k < \infty$, the group is isomorphic with $(\ZZZ_{2^k}, +)$ and generated by $1 0^{k - 1}$. $(\Sigma^\infty, \lp)$ is isomorphic with $(\ZZZ, +)$ and it is generated by $1 0^\infty$.
\end{lema}

For $u\in\Sigma^k$ and $n\in\ZZZ$ we abbreviate $u \lp n\cdot 10^{k-1}$ by $u \lp n$. We will write the inverse element of $u = u_1^k$, $1\leq k\leq \infty$, as $\lm u$, i.e.~for such a $u$ and $\lm u$ we have $u \lp (\lm u) = 0^k$. Clearly, $\lm u = v \lp 1$ where $v_i = 1 - u_i$. For $u\in\Sigma^k$, where $k < \infty$ or $u = v^{(m)}0^\infty$ with $m < \infty$, there exists $n < 2^k$ such that $0^k \lp n = u$. In fact, if $u = u_1^m 0^{k - m}$, $0\leq m < k + 1$, then
\begin{align}
\label{eq:iteracia}
n = u_1\cdot 1 + u_2\cdot 2 + \ldots + u_j \cdot 2^{j - 1} + \ldots + u_m\cdot 2^{m - 1}.
\end{align}
This can be easily checked by induction. For $w\in\Sigma^k$ and $v\in \{w \lp n\ |\ n\in\ZZZ\}$ (note that this set is the whole $\Sigma^k$ for $k < \infty$), we denote the unique integer $0\leq n < 2^k$ by $w \lm v$ satisfying $w = v\lp n$.

Let $0\leq k\leq \infty$. The dynamical system $(\Sigma^k, g)$, where $g^n(u) = u \lp n$ for $n\in\ZZZ$ and $u\in\Sigma^k$, is called the \emph{adding machine} (or \emph{odometer}). Later we write $(\Sigma^k, \lp)$ instead of $(\Sigma^k, g)$. Let $u_1^k, v_1^k\in\Sigma^k$ be such that $u_1^h = v_1^h$ and $u_{h + 1} \neq v_{h + 1}$ for some $0\leq h < k$. Then $(u \lp 1)_1^h = (v \lp 1)_1^h$ and $(u \lp 1)_{h + 1} \neq (v \lp 1)_{h + 1}$, therefore the adding machine with the metric defined by $d(u_1^k, v_1^k) = \max(\{0\}\cup\{1/2^i\ |\ 1\leq i\leq k, u_i\neq v_i\})$ is an isometry. From now on we make the assumption that this metric defines the topology on $\Sigma^k$, i.e.~sets of words of the length $k$ beginning with the same word form a basis for topology on $\Sigma^k$.

Let $1\leq k < \infty$ and $u = u_1^k \in\Sigma^k$. Sometimes it is more comfortable to think about orders of words instead of their dyadic codes. We write $\gamma(u_1^k) = 1 + u_k\cdot 1 + u_{k - 1}\cdot 2 + \ldots + u_j\cdot 2^{k - j} + \ldots u_1\cdot 2^{k - 1}$. For $u, v\in\Sigma^k$ we thus write $u < v$ whenever it holds that $\gamma(u) < \gamma(v)$. For $k = \infty$ we write $u^{(\infty)} < v^{(\infty)}$ if there is $0 < h < \infty$ with $u_1^{h} < v_1^{h}$. If $u\neq v$, such an $h$ always exists. The relations $\leq, >, \geq$ on the words are defined similarly.

We define $\pp$ to be the addition from the right to the left on $\Sigma^k$ where $1\leq k < \infty$, e.g.~$100 \pp 110 = 010$. The pair $(\Sigma^k, \pp)$ is a cyclic group with the identity element $0^k$ and the inverse element $\pminus u = v \pp 0^{k - 1} 1$ where $v_i = (1 - u_i)$ for all $1\leq i\leq k$. For $n\in\NNN$, we write $u \pp n = u \pp n\cdot 0^{k - 1}1$. Let $0\leq n < \gamma(v)$, then $v\pminus n$ is defined as the unique word $u\in\Sigma^k$ such that $u\pp n = v$.

Let the map $\cores^{\alpha}: \Sigma^\infty\to [0, 1]$ be given by
\begin{align*}
\cores^{\alpha}(u_1 u_2\ldots ) = \cores(u_1 u_2 \ldots) = (1 - \alpha)\cdot(u_1 + u_2\cdot\alpha + u_3\cdot\alpha^2 + \ldots + u_k\cdot\alpha^{k - 1} + \ldots).
\end{align*}
It can be easily checked that $\cores$ is injective for all $\alpha$. If $k < \infty$ and $u\in\Sigma^k$, then we mean $\cores(u 0^\infty)$ by $\cores(u)$.

Let $u, v\in \Sigma^k$; $\hat{u}, \hat{v}\in\Sigma^m$,  and $1\leq k, m < \infty$. We now prove that
\begin{align}
\label{eq:iff}
\gamma(u) < \gamma(v) \Leftrightarrow \gamma(u\hat{u}) < \gamma(v\hat{v}) \Leftrightarrow \cores(u\hat{u}) < \cores(v\hat{v}) \Leftrightarrow \cores(u) < \cores(v).
\end{align}
Suppose that $\gamma(u) < \gamma(v)$. Obviously, we now have $u\neq v$. Then, there are $h\geq 0$ and the words $u^{(h)}\in\Sigma^h$, $u^{(k - h - 1)}, v^{(k - h - 1)}\in\Sigma^{k - h - 1}$ such that $u = u^{(h)} 0 u^{(k - h - 1)}$ and $v = u^{(h)} 1 v^{(k - h - 1)}$. It is always true that $\alpha^{k - h} > \alpha^{k - h + 1} + \alpha^{k - h + 2} + \ldots + \alpha^k$, hence $\cores(u) < \cores(v)$. Since the proof does not depend on $k$ or on the words $u^{(k - h - 1)}, v^{(k - h - 1)}$, we can assume that $u = u^{(h)} 0 u^{(k - h - 1)}\hat{u}$ and $v = u^{(h)} 1 v^{(k - h - 1)} \hat{v}$. We can use the same arguments if we assume that $\cores(u) < \cores(v)$. The equivalences~(\ref{eq:iff}) are now proved. If $u, v\in\Sigma^\infty$, then neither $\gamma(u)$ and $\gamma(v)$ nor $u\hat{u}$, $v\hat{v}$ are defined. However, we can use the same arguments to prove that
\begin{align}
\label{eq:iffinf}
u < v \Leftrightarrow \cores(u) < \cores(v).
\end{align}

Let $u = u_1^k, v = v_1^k \in\Sigma^k$. Denote by $d_E$ the Euclidean metric on $[0, 1]$. Therefore we can define the distance of words on $\Sigma^k$ by
\begin{align*}
\dis^\alpha(u, v) = \dis(u, v) = d_E(\cores(u), \cores(v)) = (1 - \alpha)\left| \sum_{i = 1}^{k}(u_i - v_i)\alpha^{i - 1}\right|.
\end{align*}

\begin{lema}
\label{lema:f_a_sigma}
For every $u\in\Sigma^\infty$ and $n\geq 1$,
\begin{align*}
f_{\alpha}^n(\cores(u)) = \cores(u \lp n).
\end{align*}
\end{lema}
\begin{proof}
If $u = 1^\infty$, then $u\lp 1 = 0^\infty$, $\cores(u) = 1$, $\cores(u \lp 1) = 0$ and $f_{\alpha}(1) = 0$.

Suppose that $u < 1^\infty$, then there is some $h\geq 0$ such that $u = 1^h 0 \hat{u}$ where $\hat{u}\in\Sigma^\infty$. Then $\cores(u) = 1 - \alpha^h + \alpha^{h + 1}\cdot\cores(\hat{u})$, $u \lp 1 = 0^h 1 \hat{u}$ and $\cores(u \lp 1) = \alpha^h\cdot(1 - \alpha) + \alpha^{h + 1}\cdot\cores(\hat{u})$. From~(\ref{eq:iffinf}),
\begin{align*}
1 - \alpha^h = \cores(1^h 0^\infty) \leq \cores(u) \leq \cores(1^h 0 1^\infty) = 1 - \alpha^h + \alpha^{h + 1}.
\end{align*}
Using Definition~\ref{df:del}, where $j = h + 1$, we can conclude that
\begin{align*}
f_{\alpha}(\cores(u))  	& = \cores(u) - 1 + 2\alpha^h - \alpha^{h + 1} = 1 - \alpha^h + \alpha^{h + 1}\cores(\hat{u}) - 1 + 2\alpha^h - \alpha^{h + 1} = \\
                        & = \alpha^h(1 - \alpha) + \alpha^{h + 1}\cores(\hat{u}) = \cores(u \lp 1).
\end{align*}
The general case here is proved with the induction.
\end{proof}

For every $v\in\Sigma^k$ there exists $n\geq 0$ satisfying $v 0^\infty = 0^\infty \lp n$. By Definition~\ref{df:del}, $f(\cores(w)) = f_k(\cores(w))$ for every $w = v 0^\infty < 1^k0^\infty$ and $f_k(\cores(1^{k}0^\infty)) = f_k(1 - \alpha^k) = 0 = \cores(0^\infty)$. Therefore, by Lemma~\ref{lema:f_a_sigma},
\begin{align}
\label{eq:fk_map}
f_k^n(\cores(v)) = \cores(v \lp n).
\end{align}
For $f_k$, it is thus sufficient to consider words of the length $k$ and vice versa.

Set $X_0 = [0, 1]$. For $1\leq k < \infty$, let $(I_{\alpha}(u))_{u\in\Sigma^k}$ be the system of compact subintervals of $[0, 1]$ defined by
\begin{equation}
\label{eq:interval}
I_{\alpha}(u) = I(u) = [\cores(u), \cores(u) + \alpha^k].
\end{equation}
Thus, the distance of words from $\Sigma^k$ is actually the distance of left points of intervals $I(u)$ and $I(v)$. Denote $X_{\alpha, k} = X_k = \bigcup_{u\in \Sigma^k} I(u)$. If $w\in\Sigma^0$, then $I(w) = X_0$.

For every $u\in\Sigma^k$ such that $u < 1^k$ and every $x\in I(u)$, we have $f(x) = f_k(x)$. For $x\in I(1^k)$, the equality need not hold, however $f(I(1^k)) = f_k(I(1^k)) = I(0^k)$. Since 
\begin{itemize}
	\item each $I(u)$ has the same length $\alpha^k$ and
	\item the slope of the restricted maps $\restr{f_k}{I(u)}$ equals $1$,
\end{itemize}
from~(\ref{eq:fk_map}) it follows that
\begin{equation}
\label{eq:iter_int}
    f^n(I(u)) = I(u\lp n) = f_k^n(I(u))
    \qquad\text{for every}\quad
    u\in\Sigma^k, n\geq 0.
\end{equation}
The unique uncountable minimal set $X$ of $f$ is
\begin{equation*}
    X_{\alpha} = X = \bigcap_{k\geq 1} X_k.
\end{equation*}
From the latter part of Lemma~\ref{lema:grupa}, each interval $I(u)$, $u\in\Sigma^k$, is $2^k$ periodic under $f$ and $f_k$. The \emph{trajectory} of $u\in\Sigma^k$ is the $2^k$--periodic sequence $(u \lp n)_{n\geq 0}$. Similarly, the trajectory of $I(u)$ is the $2^k$--periodic sequence $I(u \lp n)_{n\geq 0}$.

\begin{figure}[!ht]
    \centering
        \includegraphics[width = 0.5\textwidth]{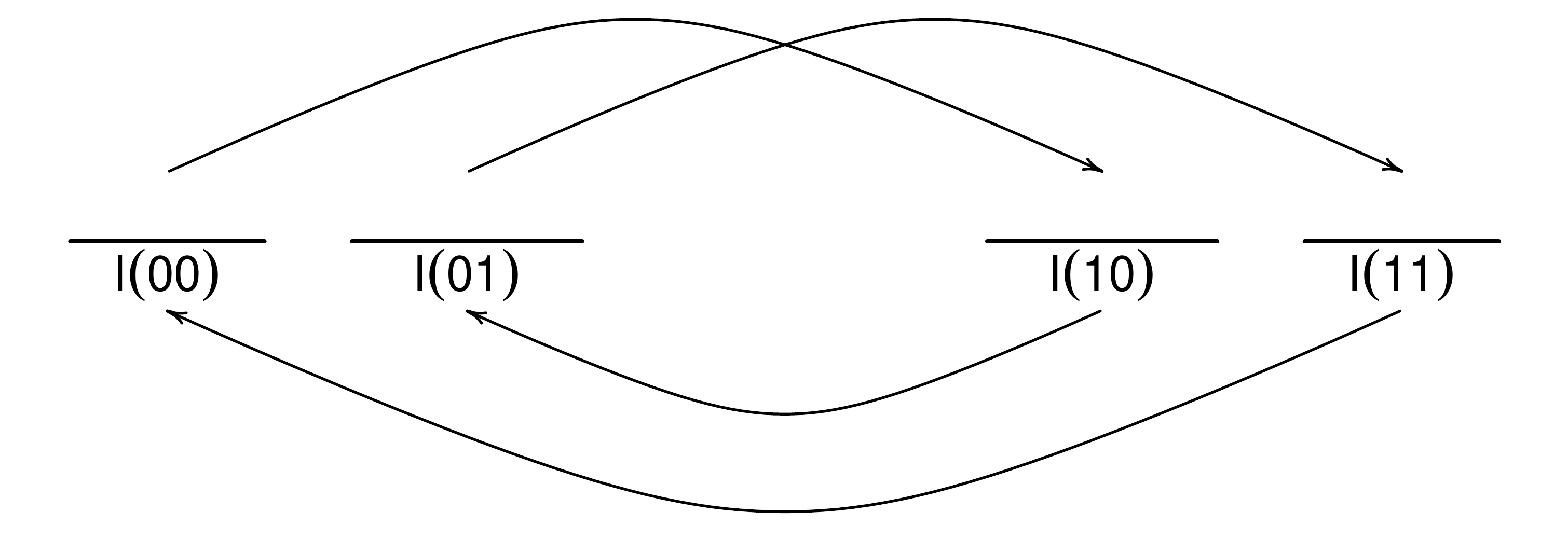}
    \caption[Adding machine]{Visualization of the trajectory of $I(00)$.}
    \label{fig:adding}
\end{figure}

\begin{lema}
\label{lema:conjugate}
The map $\restr{f_{\alpha}}{X}$ is topologically conjugate to $(\Sigma^\infty, \lp)$.
\end{lema}
\begin{proof}
By Lemma~\ref{lema:f_a_sigma}, it is sufficient to prove that $\cores$ is a continuous bijection between $X$ and $\Sigma^\infty$.

Since $\cores(0 1^\infty) < \cores(1 0^\infty)$, the map $\cores$ is injective. We have to prove that $\cores$ is also surjective. Let $x\in X$, then there is a sequence of words $(u^{(i)})$ such that $x\in I(u^{(i)})$ for every $i\geq 1$. Therefore $x\in\bigcap_{i = 1}^\infty I(u^{(i)})$. Clearly, for every $i\geq j$, the word $u^{(i)}$ begins with $u^{(j)}$ and $I(u^{(i)})\subset I(u^{(j)})$. Moreover, every $u\in\Sigma^\infty$ with the property that $\cores(u)\in I(u^{(j)})$ begins with $u^{(j)}$.

Intervals $I(u^{(i)})$ are compact and their diameters converge to zero, therefore their intersection is a singleton. Consider $u\in\Sigma^\infty$ with the property that $u_1^k = u^{(k)}$ for every $k > 0$. It follows that $\cores(u)\in I(u^{(i)})$ for every $i$ and $\cores(u) = x$.

It remains to show that $\cores$ is continuous, i.e.~for every $x\in X$ and for every neighbourhood $U_x$ of $x$ there is a neighbourhood $V_x$ of $u = \cores^{-1}(x)$ such that $\cores(V_x)\subset U_x$. The basis for the topology on $\Sigma^\infty$ is formed with sets $B(w, k) = \{v\in\Sigma^\infty\ |\ v_1^k = w_1^k\}$. Obviously, $\cores(B(w, k))\subset I(w_1^k)$. Since $U_x$ is open, there is $k > 0$ satisfying $I(u_1^k)\subset U_x$. Then $\cores(B(u, k)) \subset I(u_1^k)$.
\end{proof}

Let $u, v\in\Sigma^k$. We say that $I(u)$ and $I(v)$ are $\ve$--\emph{close} if $\dis(u, v)\leq\ve$. Analogously, $I(u)$ and $I(v)$ are $\ve$--\emph{distant} if $\dis(u, v) > \ve$. 

Let the matrix $M = (m_{i, j})$ of the size $n\times m$ be given. Let $\beta\in\{1, 2, \ldots, n\}^{n'}$, respectively $\delta\in\{1, 2, \ldots, m\}^{m'}$, where $n', m'\geq 1$. Denote by $\beta[i]$, respectively that of $\delta$ by $\delta[i]$, an $i$\emph{th} element of $\beta$. Then $M[\beta, \delta] = (m'_{i', j'})$ is a matrix of the size $n'\times m'$ with $m'_{i', j'} = m_{i, j}$ where $i = \beta[i']$ and $j = \delta[j']$. If $\beta = (i)$ and if $\delta = (j)$, then $M[\beta, \delta] = M[i, j] = m_{i, j}$.

\begin{df}
\label{df:distance_matrix}
Let $f_{\alpha, k}$ and $1\leq\ell\leq\infty$ be given. The matrix $D_{k, \alpha, \ell}(\ve)$ of the size $2^k\times 2^k$ is called the \emph{$\ell$--distance matrix} where:
	\begin{itemize}
		\item $D_{k, \alpha, \ell}(\ve)[\gamma(u), \gamma(v)] = 1$ if $\dis_{\ell}^{\alpha}(u, v)\leq \ve$ and
		\item $D_{k, \alpha, \ell}(\ve)[\gamma(u), \gamma(v)] = 0$ otherwise,
	\end{itemize}
for every $u, v\in\Sigma^k$. Instead of $D_{k, \alpha, \ell}$, we write $D_{k, \alpha}$ if $\ell = 1$.
\end{df}

By~(\ref{eq:iter_int}), the operation $\lp$ defines iterations of intervals under $f$ and $f_k$. Let $u, w\in\Sigma^k$. We show that
\begin{align}
\label{eq:gamma_scitanie}
\gamma(u\pp n) = \gamma(u) + n,\qquad \gamma(w\pminus m) = \gamma(w) - m,
\end{align}
for every $0\leq n \leq 2^k - \gamma(u)$ and $0\leq m < \gamma(w)$. This means that operations $\pp, \pminus$ define moves of intervals $I(v)$ on $X_k$. That is, e.g.~for $0^k < u < 1^k$ the interval $I(u \pp 1)$ lies to the right of $I(u)$, respectively $I(u \pminus 1)$ lies to the left of $I(u)$.

For $n = 0$ and $m = 0$, the statements hold trivially.
Suppose that $u < 1^k$, then there is $0\leq k' < k$ such that $u = v^{(k - k' - 1)}0 1^{k'}$. Then $u\pp 1 = v^{(k - k' - 1)}1 0^{k'}$ and $\gamma(u\pp 1) = \gamma(u) + 1$. Let $u\pp n' < 1^k$ and assume that $\gamma(u\pp n') = \gamma(u) + n'$. We thus have 
\begin{align*}
\gamma(u\pp n' \pp 1) = \gamma((u\pp n')\pp 1) = \gamma(u\pp n') + 1 = \gamma(u) + n' + 1. 
\end{align*}
Hence, $\gamma(u \pp n) - \gamma(v\pp n) = \gamma(u) - \gamma(v)$.
Let $v\in\Sigma^k$ be such that $v \pp m = w$, i.e.~$v = w\pminus m$. Therefore $\gamma(w) = \gamma(v\pp m) = \gamma(v) + m$ and $\gamma(w\pminus m) = \gamma(w) - m$.

\section{Distance matrix and properties of \texorpdfstring{$f_\alpha$}{fa}}
\label{del:dis_matrix}
\label{kap_dist_matrix}

The purpose of this section is to introduce patterns contained in distance matrices and also to describe the behaviour of word trajectories.

\begin{lema}
\label{lema:trajektorie_dvoch_bodov_rozsah}
Let $k > h\geq 0$ be given and $u, v\in\Sigma^k$ are such that $u_1^h = v_1^h$ and $u_{h + 1}\neq v_{h + 1}$. Then for all $x\in I(u), y\in I(v)$,
\begin{equation*}
    (1-2\alpha) \alpha^h < (d_E)_{\infty}^f(x,y) < \alpha^h.
\end{equation*}
If $u = 0^{h} u_{h + 1}^k$ and $v = 0^h v_{h + 1}^k$, then $(d_E)_{\ell}^f(x, y) = d_E(x, y)$ for $\ell < 2^h$.
\end{lema}
\begin{proof}
Note that the (Euclidean) length of $I(w)$ with $w\in\Sigma^h$ is $\alpha^h$. Since words have the first $h$ coordinates equal, there is $w\in\Sigma^h$ such that $x, y\in I(w)$, hence their distance is at most the length of this interval. The $(h + 1)$\emph{st} coordinates are distinct and therefore their distance has to be at least the length of the gap between intervals $I(w0)$ and $I(w1)$.

The interval $I(w)$ is $2^h$ periodic, and $I(w0)$ and $I(w1)$ are $2^{h + 1}$ periodic. Therefore $f^n(x), f^n(y)\in I(w\lp n)$. Without loss of generality, $f^n(x)\in I(w0\lp n)\neq I(w1 \lp n)\ni f^n(y)$ for every $n$.

If $u = 0^h u_{h + 1}^k$ and $v = 0^h v_{h + 1}^k$, then $I(u), I(v)\subset I(0^h)$. Since $I(0^h \lp 2^h) = I(0^h)$, it follows that, by Lemma~\ref{lema:grupa}, $I(0^h\lp n)\neq I(1^h)$ for $n < 2^h$. Since the slope of $f$ is one on each such $I(0^h \lp n)$, the maps $f, f^2, \ldots, f^{2^h - 1}$ are isometries restricted to $I(0^h)$. By definition of $(d_E)_\ell$, the proof is complete.
\end{proof}

Each point in $[0, 1]$ is either eventually periodic under $f_{\alpha}$ or belongs to $X_k$ for every $k > 0$, after a finite number of iterations. No (eventually) periodic points form a scrambled set. From the previous lemma, no distinct points from $X$ form a scrambled set either. Now, let one point be from $X$ and the other be an eventually periodic point $y$. In this case there is $k > 0$ such that the orbit of $y$ does not intersect $X_k$. Thus, it has positive distance from this set. The orbit of each point from $X$ lies in $X\subset X_k$, hence the pair under consideration is not a Li-York pair. Therefore Lemma~\ref{lema:trajektorie_dvoch_bodov_rozsah} provides an easy argument that no $f_{\alpha}$ is Li-Yorke chaotic.

The weaker version of the following lemma was proven in~\cite{CK-delahaye}. Later we will need a more specific version presented here.

\begin{lema}
\label{lema:per_del}
There is a unique $2^k$--periodic orbit under $f$ in every set $X_{k}\setminus X_{k + 1}$. Moreover, if $x$ is one of such points, then $f^i(x) = f_k^i(x)$ for every $i\in\ZZZ$.
\end{lema}
\begin{proof}
By~(\ref{eq:iter_int}), $f^{2^k}(I(0^k)) = I(0^k)$ and $f^{i}(I(0^k))\neq I(0^k)$ for every $k\geq 0$ and $0 < i < 2^k$. Thus, by Brouwer fixed-point theorem, in $I(0^k)$ there exists a periodic point under $f^{2^k}$ with the smallest period $2^k$. Since $f^{2^k}(I(0^{k + 1})) = I(0^k 1)\neq I(0^{k + 1})$, the point lies in $I(0^k)\cap (X_k\setminus X_{k + 1})$.

Suppose that there are two disjoint periodic orbits in $X_k\setminus X_{k + 1}$. Let $x, y\in I(0^k)\cap (X_k\setminus X_{k + 1})$ belong to these. From the previous lemma, $(d_E)_{\ell}(x, y) = d_E(x, y)$ for $\ell < 2^k$. Since $x, y\in X_k\setminus X_{k + 1}$ are periodic and both $X_k$, $X_{k + 1}$ are $f$--invariant, $f^{m}(x), f^m(y)\in X_k\setminus X_{k + 1}$ for every $m\in\ZZZ$. Put $J = I(1^k)\cap (X_k\setminus X_{k + 1})$. Then $f^{2^k - 1}(x), f^{2^k - 1}(y)\in J$ and the function $\restr{f}{J}$ is linear with the slope less than $-1$. Thus, $d_E(f^{2^k}(x), f^{2^{k}}(y)) > d_E(x, y)$, contrary to the periodicity of points.

Since $f^n(x)\in I(0^k \lp n)\neq I(0^k)$ for $0\leq n < 2^k$ and since the definition of $f_k$, it follows that $f^n(x) = f_k^n(x)$. The period of $x$ is $2^k$, therefore it is sufficient to show that $x = f_k^{2^k}(x)$. For each $u\in\Sigma^k$, the map $\restr{f_k}{I(u)}$ is linear with slope $1$ and $f_k^{2^k}(I(u)) = I(u)$. Therefore, every $f_k^{2^k}$ restricted to $X_k$ is an identity.
\end{proof}

The previous lemma could be proven slightly more elegantly using the graph of the map $f_{\alpha}$. It follows from the graphically obvious fact that there is a unique fixed point and the fact that the graph of $\restr{f_{\alpha}}{[1 - \alpha, 1]}$ is the smaller copy of $f_{\alpha}$. More precisely, the map $\restr{f_{\alpha}}{[0, \alpha]}$ is linear with the slope of the line equal to one and maps onto $[1 - \alpha, 1]$. Hence $\restr{f_{\alpha}^2}{[1 - \alpha, 1]}$ is conjugate to $f_{\alpha}$. Therefore the number of fixed points is the same in both the cases. The second part of the lemma would be proven similarly. However, we chose more algebraic way to prove the lemma to keep the paper more consistent. 

Since $X_0 = [0, 1]$ and the set $X = \bigcap_{k = 0}^{\infty}X_k$ contains no periodic points, it can be seen, that if $x$ is $f$--periodic point of the period $2^k$, then $x\in X_k\setminus X_{k + 1}$ and in $[0, 1]$ are no $f$--periodic points of other period than $2^k$, $k\geq 0$. Thus, each $f_{\alpha}$ is indeed of $2^\infty$ type.

\begin{dosledok}
\label{cor:per_vyskyt}
The set of periods for $f$ is $\{2^k\ |\ k\geq 0\}$. If $x$ is of the period $2^k$, then $x\in X_k\setminus X_{k + 1}$.
\end{dosledok}

The point $0$ is not $f$--periodic but it is $f_k$--periodic for every $k\geq 0$. In fact, left points of intervals $I(u)$, $u\in\Sigma^k$, form the orbit $\Orb_{f_k}(0)$. We conclude that analysis of distances of trajectories of words is the same as that of the recurrence of $f_k$--trajectories of $0$.

\begin{lema}
\label{lema:najvacsia_vzdialenost_cyklus_intervalov}
Let $h, k > 0$ and $u, v\in\Sigma^{k + h + 2}$ satisfy $u_1^h = v_1^h$ and $u_{h + 1}\neq v_{h + 1}$. If $n\geq 0$ is such that
\begin{align*}
u \lp n = 0^{k + h + 2},\ v \lp n = 0^h 1 1 w\qquad\text{or}\qquad
u \lp n = 0^h 1 1 w,\ u \lp n = 0^{k + h + 2}
\end{align*}
for some $w\in\Sigma^k$, then
\begin{align*}
\dis_{\infty}(u, v) = \dis(0^{k + h + 2}, 0^k 1 1 w).
\end{align*}
Such an $n$ always exists and is unique up to the period $2^{k + h + 2}$.
\end{lema}
\begin{proof}
Without loss of generality, we can assume that $u = u_1^h 0 u_{h + 2} u^{(k)}$ and $v = u_1^h 1 v_{h + 2} v^{(k)}$. Let $0\leq n, m < 2^{k + h + 2}$ be such that $n\cdot 1 0^{k + h + 1} = \lm u$ and $m\cdot 1 0^{k - 1} = \lm v$. Assume that $u_{h + 2} \neq v_{h + 2}$. We thus have $u \lp (\lm u) = 0^{k + h + 2}$ and $v \lp (\lm u) = 0^h 1 w$ where $w = v_{h + 2} v^{(k)} \lp (\lm u_{h + 2} u^{(k)})$. Since the first coordinates of a word and its inverse are identical and, by assumption, $u_{h + 2} \neq v_{h + 2}$, it follows that $w_1 = 1$ and $n$ is the desired iteration. It can be shown similarly that $m$ does not satisfy our assumptions. If $u_{h + 2} = v_{h + 2}$, then $n$ does not meet assumptions, but $m$ does it so. In this case, $m$ is the desired iteration. The inverse element is unique in a group, therefore $n, m < 2^{k + h + 2}$ are unique as well.

It remains to show that the distance is maximal. From Lemma~\ref{lema:grupa}, the words $u, v$ are $2^{k + h + 2}$ periodic, therefore, without loss of generality, we can assume that $u = 0^h 1 1 a$, $a\in\Sigma^k$, and $v = 0^{k + h + 2}$. From Lemma~\ref{lema:trajektorie_dvoch_bodov_rozsah}, the distance changes only after $m\cdot 2^h$ iterations where $m\geq 1$, i.e.~iterations $u\lp n$ and $v\lp n$ are such that $n\cdot 1 0^{k + h + 1} = 0^{h} b_0^1 b_0^2 b$ where $b = b_1^k\in\Sigma^k$. Let $a = a_1^k\in\Sigma^k$. The distance $\dis(0^{k + h + 2}, 0^h 1 1 a)$ is at least $\alpha^h - \alpha^{h + 2}$. We have to show that $\dis(0^{k + h + 2}\lp 0^h b_0^1 b_0^2 b, 0^h 1 1 a \lp 0^h b_0^1 b_0^2 b)\leq \alpha^h - \alpha^{h + 2}$ for every $b_0^1 b_0^2 b\in\Sigma^{k + 2}$.

\noindent
There are four possibilities for $b_0^1, b_0^2$. For a better visualization of iterations, see Figure~\ref{fig:adding}.
\begin{itemize}
    \item $b_0^1 = 1, b_0^2\in\{0, 1\}$. In these cases, the maximal distance is $\alpha^h - \alpha^{h + 2} - \alpha^{h + 1}(1 - 2\alpha) < \alpha^h - \alpha^{h + 2}$.
    \item $b_0^1 = 0, b_0^2 = 1$. The maximal distance of points in such intervals is $\alpha^h(1 - 2\alpha) + 2\alpha^{h + 2} < \alpha^h - \alpha^{h + 2}$.
    \item $b_0^1 = b_0^0 = 0$. This is the only option where it is not that easy to see that the distance cannot increase. This iteration is such that $I(0^{k + h + 2})$ maps into $I(0^{k + h + 2} \lp 0^{h + 2} b) = I(0^{h + 2} b)$, so does $I(0^{h} 1 1 a)$ into $I(0^{h} 1 1 w \lp 0^{h + 2} b) = I(0^{h} 1 1 c_1 c_2 \ldots c_k)$, where $c_i$s satisfy
            \begin{equation}
                \label{eq:modula}
                \begin{split}
                    a_1 + b_1 + q_{0}&= c_1 + q_1\\
                    a_2 + b_2 + q_1 &= c_2 + q_2\\
                    \cdots\\
                    a_k + b_k + q_{k - 1} &= c_k + q_k,
                \end{split}
            \end{equation}
        where $q_0 = 0$ and $q_i$ is $0$ if $a_i + b_i + q_{i - 1} \leq 1$, otherwise $q_i = 1$. The number $c_i\in\{0, 1\}$ is congruent to $a_i + b_i + q_{i - 1}$ modulo $2$.

        The distance of left points of resulting intervals is $\alpha^h(1 - \alpha)(1 + \alpha + (c_1 - b_1)\alpha^{2} + \ldots + (c_k - b_k)\alpha^{k + 1})$; the distance of initial left points is $\alpha^{h}(1 - \alpha)(1 + \alpha + a_1\alpha^{2} + \ldots + a_k\alpha^{k + 1})$.
            \begin{itemize}
                \item[a)] For all $i:\ c_i - b_i = a_i$, the distance is the same.
                \item[b)] There is $n\geq 1$ such that for all $i < n:\ c_i - b_i = a_i$ and $c_n - b_n < a_n$. Therefore $a_n - (c_n - b_n)\geq 1$ and
                    \begin{align*}
                        (c_n - b_n) + &(c_{n + 1} - b_{n + 1})\alpha + \ldots + (c_k - b_k)\alpha^{k - n} \leq\\
												&\leq (c_n - b_n) + \alpha + \ldots + \alpha^{k - n} < (c_n - b_n) + 1 \leq a_n \leq\\
												&\leq a_n + a_{n + 1}\alpha + \ldots + a_k\alpha^{k - n}.
                    \end{align*}
                \item[c)] There is $n\geq 1$ such that for all $i < n:\ c_i - b_i = a_i$ and $c_n - b_n > a_n$. Therefore $c_n - b_n = 1, a_n = b_n = 0, q_{n - 1} = c_n = 1$ and $n\geq 2$. Moreover, for $i < n: c_i = \max\{a_i, b_i\}$ and $a_i\cdot b_i = 0$. Thus, $q_i = 0$ for all $i < n$, contrary to $q_{n - 1} = 1$.
            \end{itemize}
\end{itemize}
The proof is finished, since no other case can occur.
\end{proof}

From the previous lemma, whole trajectories of intervals $I(u)$ and $I(v)$ are $\ve$--close if and only if corresponding intervals $I(0^{k + h + 2})$ and $I(0^{h} 1 1 w)$ are $\ve$--close.

\begin{dosledok}
\label{dosledok:pocet_blizke_orbity}
Let $h\geq 0$ be such that $\alpha^{h + 1} < \ve \leq \alpha^h$ and $k \geq h + 2$. Let $u\in\Sigma^{k - h - 2}$ be such that $\dis(0^k, 0^h 1 1 u)\leq\ve$ and $\dis(0^k, 0^{h'} 1 1 v) > \ve$ for every $h'\leq h$ and $v\in\Sigma^{k - h' - 2}$ satisfying $0^h 1 1 u < 0^{h'} 1 1 v$. Then there are exactly $2^{k}(2\gamma(0^h 1 1 u) - 2^{k - h})$ pairs $u', v'\in\Sigma^k$ with $\dis_{\infty}(u', v')\leq \ve$.
\end{dosledok}

The next lemma shows that if the approximation $f_k$ of $f$ gets better, the ratio of pairs of words $u, v$ such that $\dis(u, v)\leq \ve$ and $\dis_{\infty}(u, v) > \ve$ to all pairs of words of the corresponding length cannot get much smaller. In fact, if such ratio for an approximation $f_k$ is $c/2^{2k}$, the ratio for some better approximation has to be at least $c/2^{2k + 1}$.

\begin{lema}
\label{lema:pocet_daleke_bodky}
Let $k\geq 1$ and $u, v\in\Sigma^k$ be the words satisfying $\dis(u, v)\leq \ve$. Then for every $m > 0$ there is at least $2^{m - 1}(1 + 2^m)$ pairs of words $\hat{u}, \hat{v}\in\Sigma^m$ such that $\dis(u\hat{u}, v\hat{v})\leq\ve$. In addition, if $\dis_{\infty}(u, v) > \ve$, then for all $m\geq 0$ and $\hat{u}, \hat{v}\in\Sigma^m$ the distance $\dis_{\infty}(u\hat{u}, v\hat{v}) > \ve$.
\end{lema}
\begin{proof}
Without loss of generality, we may assume that $u < v$. Then $\dis(u, v) = \cores(v) - \cores(u)$. From~(\ref{eq:iff}), we have $u\hat{u} < v\hat{v}$ for every $\hat{u}, \hat{v}\in\Sigma^m$ and therefore
\begin{align*}
\dis(u\hat{u}, v\hat{v}) 	&= \cores(v\hat{v}) - \cores(u\hat{u}) = \cores(v) - \cores(u) + \alpha^k\cdot(\cores(\hat{v}) - \cores(\hat{u})) = \\
													&= \dis(u, v) + \alpha^k(\cores(\hat{v}) - \cores(\hat{u})).
\end{align*}
Clearly, there are $\gamma(\hat{u})$ words $\hat{v}$ with the property that $\hat{v}\leq\hat{u}$. Then, by~(\ref{eq:iff}), $\cores(\hat{v}) \leq \cores(\hat{u})$ and, from the identity above, $\dis(u\hat{u}, v\hat{v})\leq \dis(u, v) \leq \ve$. Therefore there are at least $\sum_{\hat{u}\in\Sigma^m} \gamma(\hat{u}) = 2^{m - 1}(1 + 2^m)$ pairs of words from $\Sigma^m$ which satisfy $\dis(u\hat{u}, v\hat{v})\leq \ve$.

From Lemma~\ref{lema:najvacsia_vzdialenost_cyklus_intervalov}, the distance $\dis_{\infty}(u, v) = \dis(0^k, 0^{h} 1 1 w)$ where $h\geq 0$ is such that $u_1^{h} = v_1^{h}$ and $u_{h + 1}\neq v_{h + 1}$ and $w$ is a word of the length $k - h - 2$. Without loss of generality, assume that $u \lp n = 0^k$ and $v \lp n = 0^h 1 1 w$, i.e.~$v \lp (\lm u) = 0^h 1 1 w$. Then $u\hat{u}\lp (\lm u\hat{u}) = 0^{k + m}$ and $v\hat{v} \lp (\lm u\hat{u}) = 0^h 1 1 w \hat{w}$ where $\hat{w}\in\Sigma^m$. Since $\dis(0^{k + m}, 0^h 1 1 w 0^m) = \dis(0^k, 0^h 1 1 w) > \ve$, it follows that $\dis(0^{k + m}, 0^h 1 1 w \hat{w}) > \ve$ for every $\hat{w}\in\Sigma^m$.
\end{proof}

The next lemma shows that if $\alpha > 1/3$, then the pairs of words from the previous lemma always exist. Informally put, there are pairs of intervals which are $\ve$--close, but after a finite number of iterations they are $\ve$--distant.

\begin{lema}
\label{lema:existence_pair}
Let $1/3 < \alpha < 1/2$ and $\ve < 1$. Then there are $k > 0$ and $u, v\in\Sigma^k$ such that $\dis(u, v)\leq\ve$ and $\dis_{\infty}(u, v) > \ve$.
\end{lema}
\begin{proof}
Let $h \geq 0$ be such that $\alpha^{h + 1} < \ve \leq \alpha^h$. If $\ve < \alpha^h$, let $k > 0$ satisfy
\begin{align*}
\alpha^h(1 - 2\alpha) + \alpha^k \leq\ve < \alpha^h - \alpha^k,
\end{align*}
i.e.~let intervals $I(0^{h + 1} 1^{k - h - 1})$ and $I(0^h 1 0^{k - h - 1})$ be $\ve$--close, and $I(0^k)$ and $I(0^h 1^{k - h})$ be $\ve$--distant. If $\ve = \alpha^h$, then from assumption, $h > 0$. In this case, let $k$ satisfy
\begin{align*}
\alpha^{h - 1}(1 - 2\alpha) + \alpha^k \leq\alpha^h,
\end{align*}
i.e.~let intervals $I(0^h 1^{k - h})$ and $I(0^{h - 1} 1 0^{k - h})$ be $\ve$--close.

First, assume that $\ve < \alpha^h - \alpha^k$, i.e.~intervals $I(0^k)$ and $I(0^h 1^{k - h})$ are $\ve$--distant. Consider $u = 0^{h + 1} 1 ^{k - h - 1}$ and $v = 0^h 1 0^{k - h - 1}$. From assumptions, $\dis(u, v) = \alpha^h(1 - 2\alpha) +  \alpha^k \leq \ve$. Let $w\in \Sigma^{k - h - 2}$ be such that $\dis_{\infty}(u, v) = \dis(0^k, 0^h 1 1 w)$. If $\dis(0^k, 0^h 1 1 w) > \ve$, then the desired pair of words exists. Otherwise, set $u = 0^k$. Thus
\begin{align*}
\dis(u, v) = \dis(0^k, 0^h 1 0^{k - h - 1})\leq \dis(0^k, 0^h 1 1 w)\leq \ve.
\end{align*}
Let $n$ be such that $n\cdot 1 0^{k - 1} = 0^h 1^{k - h}$, then $u \lp n = 0^h 1^{k - h}$ and $v \lp n = 0^k$. Thus, $\dis_{\infty}(u, v) = \dis(0^k, 0^h 1^{k - h}) > \ve$.

Let $\ve = \alpha^h$. Put $u = 0^h 1^{k - h}$ and $v = 0^{h - 1} 1 0^{k - h}$. Then $\dis(u, v) = \alpha^{h - 1}(1 - 2\alpha) + \alpha^k\leq \ve$. For all $w\in\Sigma^{k - h - 1}$ we have $\dis(0^k, 0^{h - 1} 1 1 w)\geq \alpha^{h - 1} - \alpha^{h + 1} > \alpha^h = \ve$. Therefore, from Lemma~\ref{lema:najvacsia_vzdialenost_cyklus_intervalov}, it is true that $\dis_{\infty}(u, v) > \ve$, and $u, v$ are those desired words.
\end{proof}

Let $i\geq 1$ and $j\geq 0, m\geq 0$ be such that $i = 2^{j}\cdot (2m + 1)$. Then we call the sequence $(a_i)_{i\in\NNN}$, given by $a_i = (1/\alpha)^j$, the \emph{gap sequence}.

Let $u\in \Sigma^k$ and $u\neq 1^k$. Denote the gap between intervals $I(u)$ and $I(u\pp 1)$, that is the distance between the right point of $I(u)$ and the left point of $I(u \pp 1)$, by $b_{\gamma(u)|k}$. Since $I(v) = [\cores(v), \cores(v) + \alpha^k]$ for each $v\in\Sigma^h$, the gap between $I(u0)$ and $I(u1)$ is $\alpha^{k}(1 - 2\alpha)$. The original gap between $I(u)$ and $I(u\pp 1)$ is a new gap between $I(u1)$ and $I((u\pp 1)0)$. Therefore $b_{\gamma(u1)|k + 1} = b_{\gamma(u)|k}$ and $b_{\gamma(u0)|k + 1} = \alpha^{k}(1 - 2\alpha)$. We have $\gamma(u1) = 2\cdot\gamma(u)$ and $\gamma(u0) = \gamma(u1) - 1$.

If $\gamma(v^{(k)})$ is odd, the gap $b_{\gamma(v^{(k)})|k} = \alpha^{k - 1}(1 - 2\alpha) = \alpha^{k - 1 - 0}(1 - 2\alpha)$. If $v^{(k)} = v^{(k - 1)}1$ and if $\gamma(v_1^{k - 1})$ is odd, then $b_{\gamma(v^{(k)})|k} = b_{v^{(k - 1)}|k - 1} = \alpha^{k - 2}(1 - 2\alpha) = \alpha^{k - 1 - 1}(1 - 2\alpha)$. Similarly, if $v^{(k)} = v^{(k')}1^{k - k'}$ where $\gamma(v^{(k')})$ is odd, then $b_{\gamma(v^{(k)})|k} = b_{v^{(k')}|k'} = \alpha^{k' - 1}(1 - 2\alpha) = \alpha^{k - 1 - (k - k')}(1 - 2\alpha)$.

Since $u\neq 1^k$, there exists $k'\geq 0$ such that $u = v^{(k - k' - 1)} 0 1^{k'}$. By an easy computation, $\gamma(u) = 2^{k'} + n\cdot 2^{k' + 1}$ where $n\geq 0$ depends only on $v^{(k - k' - 1)}$. The number $\gamma(v^{(k - k' - 1)}0) = 1 + n\cdot 2$ is odd regardless of $v^{(k - k' - 1)}$. Therefore $a_{\gamma(u)} = \alpha^{-k'}$ and
\begin{align*}
b_{\gamma(u)|k} &= b_{\gamma(v^{(k - k' - 1)}0 1^{k'})|k} = b_{v^{(k - k' - 1)}0|k - k'} = \alpha^{k - 1 - k'}(1 - 2\alpha) = \\
								&= a_{\gamma(u)}\cdot\alpha^{k - 1}(1 - 2\alpha).
\end{align*}
It follows that for $u, v\in\Sigma^k$ where $u\leq v$,
\begin{equation}
\label{eq:gap_seq_dis}
\dis(u, v) = (\gamma(v) - \gamma(u))\cdot\alpha^k + \alpha^{k - 1}\cdot(1 - 2\alpha)\cdot\sum_{i = \gamma(u)}^{\gamma(v) - 1}a_i.
\end{equation}

Let $(a_i)$ be a gap sequence. Fix $m > 0$, $1\leq s < 2^m$ and $h\geq 0$. Let $s = 2^j(2i + 1)$ for some $i, j\leq 0$. Since $2^m > 2^j(2i + 1)$, we know that $2^{m - j}\geq 2i + 2$. Thus $2^{m - j - 1} - i - 1 \geq 0$ and
\begin{align*}
h 2^{m} + s  &= 2^j (2\cdot(h 2^{m - j - 1} + i) + 1),\\
h 2^{m} - s  &= 2^j (2\cdot(h 2^{m - j - 1} - i - 1) + 1).
\end{align*}
Therefore
\begin{equation*}
\label{lema:a_s}
    a_{s} = a_{h\cdot 2^{m} \pm s}.
\end{equation*}
Note that for all $m' > m > 0$ and $h' > 0$ there is $h > 0$ such that $h'\cdot 2^{m'} = h\cdot 2^m$, hence
\begin{align*}
a_s = a_{h_1\cdot 2^m + h_2\cdot 2^{m + 1} + h_3\cdot 2^{m + 2} + \ldots + h_j\cdot 2^{m + j - 1} \pm s}
\end{align*}
for all $j\geq 1$ and $h_1, h_2, \ldots, h_j \geq 0$.

Put $A(m, n) = \sum_{i = m}^{m + n - 1} a_i$ where $(a_i)$ is a gap sequence. The following lemma states that the sum of the first $n$ elements of gap sequence, i.e.~$A(1, n)$, is the smallest among all sums of $n$ successive elements.

\begin{lema}
\label{lema:postupnost_najmensie}
Let $(a_i)$ be a gap sequence. Then for $n, s > 0$,
\begin{align*}
A(1, n) \leq A(s, n).
\end{align*}
\end{lema}
\begin{proof}
For every $\alpha$, the ratio $1/\alpha > 1$. Hence $a_{2^k} > a_{2^l}$ for every $k > l$.

Let $n = 2^j (2 m' + 1)$. First, consider $m' = 0$, i.e.~$n=2^j$ for some $j\geq 0$. If $j = 0$, the statement holds trivially, therefore assume that $j\geq 1$. Then for some $m \geq 0$ and $h < 2^j$,
\begin{align*}
A(s, 2^j) = A(2^m - h, 2^j).
\end{align*}
Let $m$ be the maximal such integer. For every $j$, we have $a_{2^j + k_1\cdot 2^{j + 1}} = a_{2^j}, a_{2^{j + 1} + k_2\cdot 2^{j + 2}} = a_{2^{j + 1}}, \ldots, a_{2^{j + t - 1} + k_t\cdot 2^{j + t}} = a_{2^{j + t - 1}}, \ldots$ where $k_i\geq 0$. Hence $a_{k\cdot 2^j} \geq a_{2^j}$ for all $k\geq 1$. Then for every $s > 0$, there is $0\leq k' < 2^j$ with the property that $a_{s + k'} = a_{k\cdot 2^j} \geq a_{2^j}$ for some $k\geq 1$. From that, $2^j\leq 2^m$.

From previous Lemma, $a_{2^m - i} = a_{2^m + 2^j - i}$ for every $i\leq h$. Put $\sum_{i = k}^l a_i = 0$ for every $k > l$. From that,
\begin{align*}
A(2^m - h, 2^j) & = A(2^m - h, h) + A(2^m, 2^j - h) = A(2^m, 2^j - h) + A(2^m + 2^j - h, h) = \\
				& = A(2^m, 2^j) = a_{2^m} + A(1, 2^j - 1).
\end{align*}
It follows from the above that $a_{2^m}\geq a_{2^j}$ and $A(2^m - h, 2^j) \geq A(1, 2^j)$.

Let $m > 0$ and $j\geq 0$. There are unique integers $0\leq j = j_1 < j_2 < \ldots < j_k$, $k\geq 2$ satisfying $n = 2^{j_k} + 2^{j_{k - 1}} + \ldots + 2^{j_1}$. Finally,
\begin{align*}
A(s, n)	& = A(s, 2^{j_k}) + A(s + 2^{j_k}, 2^{j_{k - 1}}) + \ldots + A(s + 2^{j_k} + 2^{j_{k - 1}} + \ldots + 2^{j_2}, n) \geq \\
		& \geq A(1, 2^{j_k}) + A(1, 2^{j_{k - 1}}) + \ldots + A(1, 2^{j_i}) = \\
		& = A(1, 2^{j_k}) + A(2^{j_k} + 1, 2^{j_{k - 1}}) + \ldots + A(2^{j_k} + 2^{j_{k - 1}} + \ldots + 2^{j_2} + 1, 2^{j_1}) =\\
		& = A(1, n).
\end{align*}
The lemma has now been proved.
\end{proof}

From the previous lemma, we know which sums of parts of gap sequence are the smallest ones. The next lemma shows that the sum of short enough part of gap sequence, beginning with some power of two, is the biggest among all sums of equally long parts of sequence containing some multiple of this power.

\begin{lema}
\label{lema:postupnost_najvacsie}
Let $(a_i)$ be a gap sequence. Let $m \geq 0$, $s > m$, $t\geq 0$, $2^{m} < n \leq 2^s$ and $0\leq h < 2^m$, then
\begin{align*}
A(2^m, n - 2^m) \geq A(2^m + t\cdot 2^{s} - h, n - 2^m).
\end{align*}
\end{lema}
\begin{proof}
By~(\ref{lema:a_s}), $a_{2^m + t\cdot 2^s - q_1} = a_{2^m - q_1}$ and $a_{t\cdot 2^s + q_2} = a_{q_2}$ for every $0\leq q_1 < 2^m$ and $0 < q_2 < 2^s$. From assumptions, $1\leq n - h - 1 < 2^s$. Therefore
\begin{align*}
A(2^m + t\cdot 2^s - h, n - 2^m) = A(2^m - h, n - 2^m).
\end{align*}
Hence, we can conclude that $t = 0$.

If $h = 0$, the proof is complete. Assume that $h > 0$, then
\begin{align*}
&A(2^m, n - 2^m) = A(2^m, n - 2^m - h) + A(n - h, h)\qquad\text{and}\qquad \\
&A(2^m - h, n - 2^m) = A(2^m - h, h) + A(2^m, n - 2^m - h).
\end{align*}
It remains to show that $A(2^m - h, h) \leq A(n - h, h)$. Since $0 < h < 2^m$, the first sum is equal to $A(1, h)$. Applying Lemma~\ref{lema:postupnost_najmensie} we complete the proof.
\end{proof}

\noindent
For $i > 1$, denote by $s_i$ an integer satisfying $2^{s_i} < i \leq 2^{s_i + 1}$. We say that the zero-one matrix $M$ of dimension $2^k\times 2^k$ contains
	\begin{itemize}
		\item \emph{pattern $A_0$} if $M[i, j] = 0$ implies
            \begin{align*}
                M[i + n, j] = M[i, j + m] = 0
            \end{align*}
            for $n \leq 2^k - i, m \leq 2^k - j$;
		\item \emph{pattern $A_1$} if $M[i, j] = 1$ implies
            \begin{align*}
                M[i - n, j] = M[i, j - m] = 1
            \end{align*}
            for $n \leq i - j, m \leq j - i$;
		\item \emph{pattern $B_0$} if $M[1, j] = 0$ implies
            \begin{align*}
                M[1 + n, j + n] = 0
            \end{align*}
            for all $n \leq 2^k - j$;
		\item \emph{pattern $B_1$} if $M[1, j] = 1$ implies
            \begin{align*}
                M[1 + h2^{s_j + 1}, j + h2^{s_j + 1}] = M[1 + (h + 1)2^{s_j + 1} - j, (h + 1)2^{s_j + 1}] = 1
            \end{align*}
            for all $0\leq h < 2^{k - s_j - 1}$.
		\item \emph{pattern $C_0$} if $M[2^m, j] = 0$ implies
            \begin{align*}
                M[2^m + h2^{s_j + 1}, j + h2^{s_j + 1}] = M[1 + (h + 1)2^{s_j + 1} - j, 1 + (h + 1)2^{s_j + 1} - 2^m] = 0
            \end{align*}
            for all $0 \leq m \leq s_j, 0\leq h < 2^{k - s_j - 1}$;
		\item \emph{pattern $C_1$} if $M[2^m, j] = 1$ implies
            \begin{align*}
                M[2^m + h\cdot 2^{s_j + 1} - n, j + h\cdot 2^{s_j + 1} - n] = 1
            \end{align*}
            for all $0 \leq m \leq s_j, 0\leq h < 2^{k - s_j - 1}, n < \min\{2^m, j - 2^s_j\}$.
	\end{itemize}

\begin{figure}[!ht]
  \centering
  \includegraphics[width=0.25\textwidth]{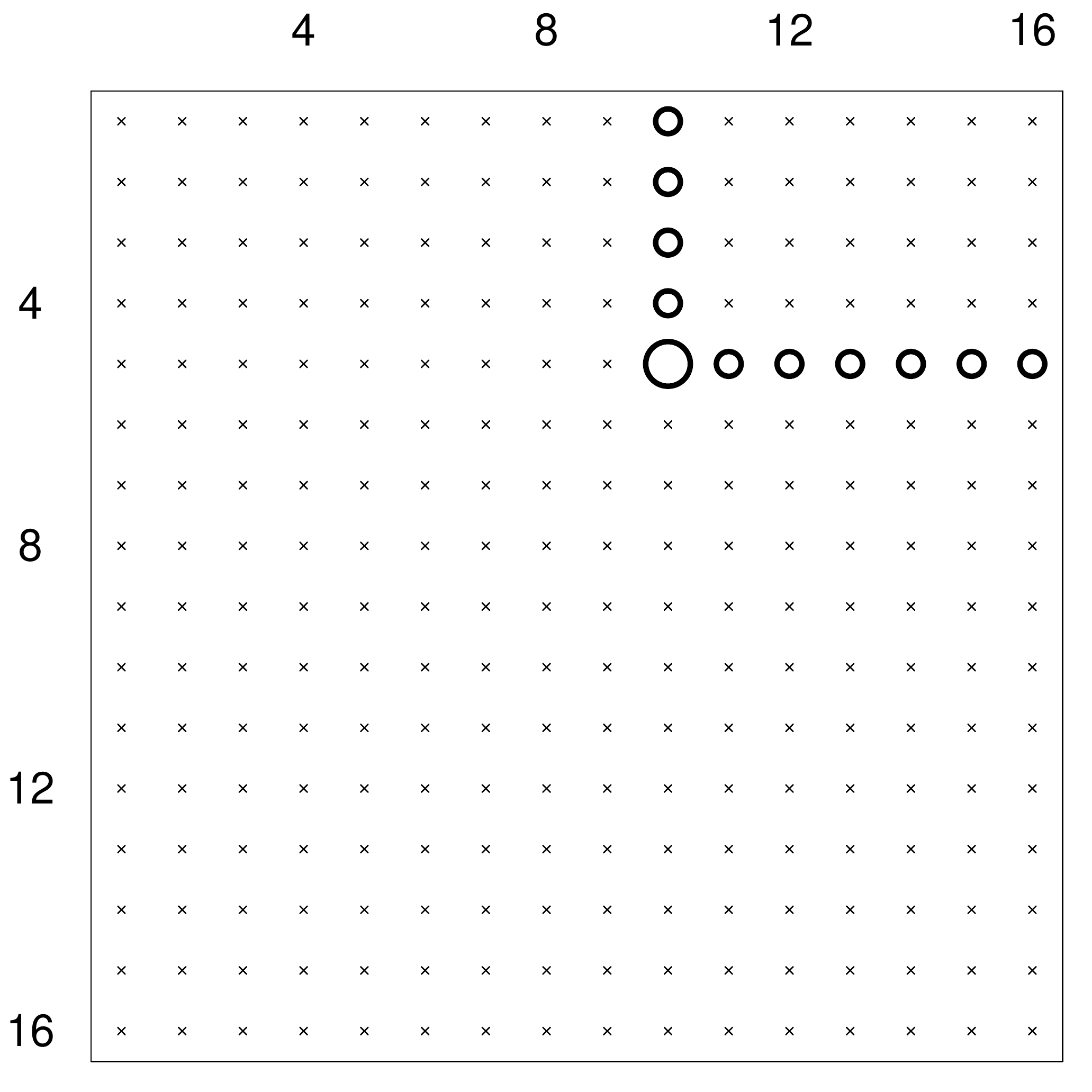}
  \includegraphics[width=0.25\textwidth]{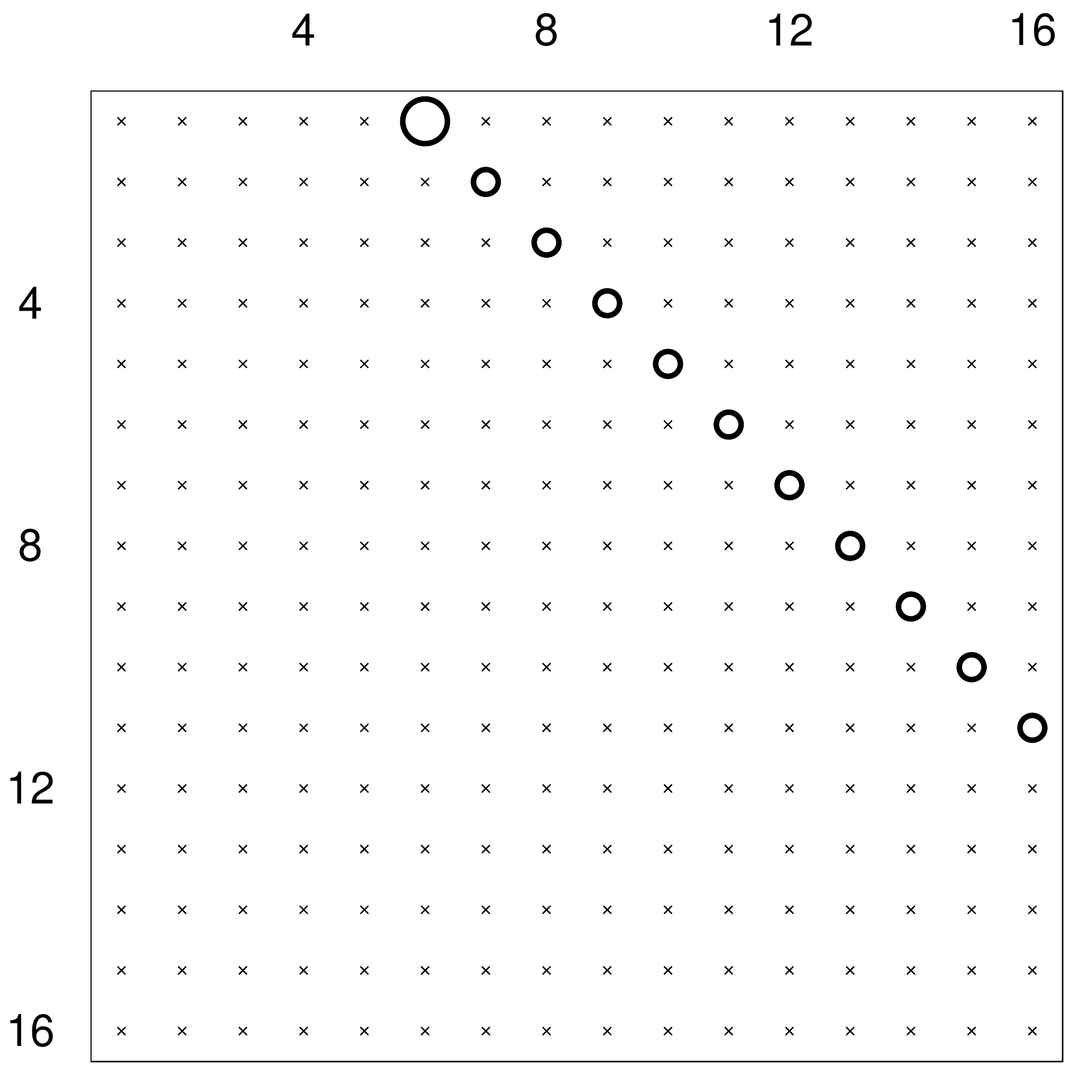}
  \includegraphics[width=0.25\textwidth]{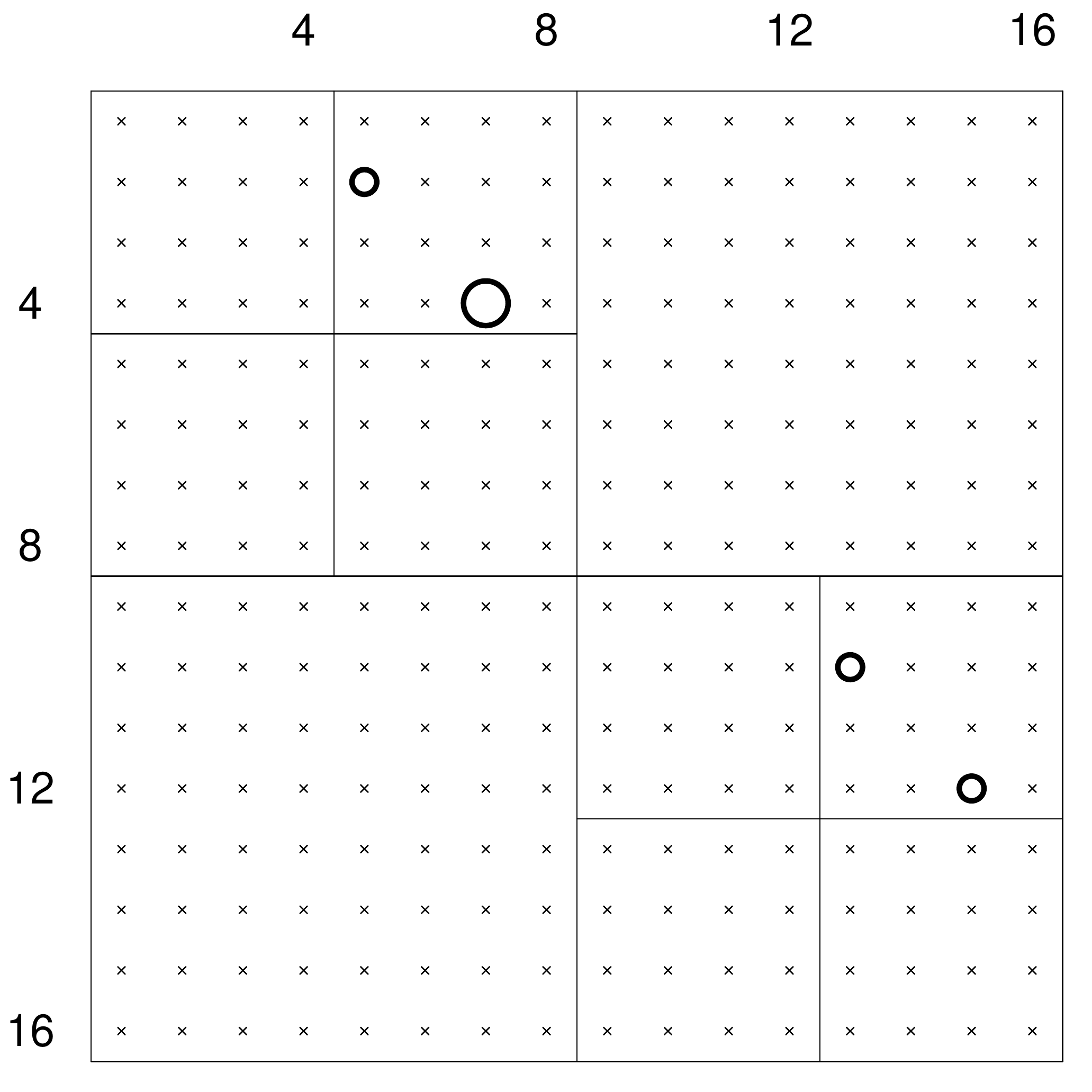}\\
  \includegraphics[width=0.25\textwidth]{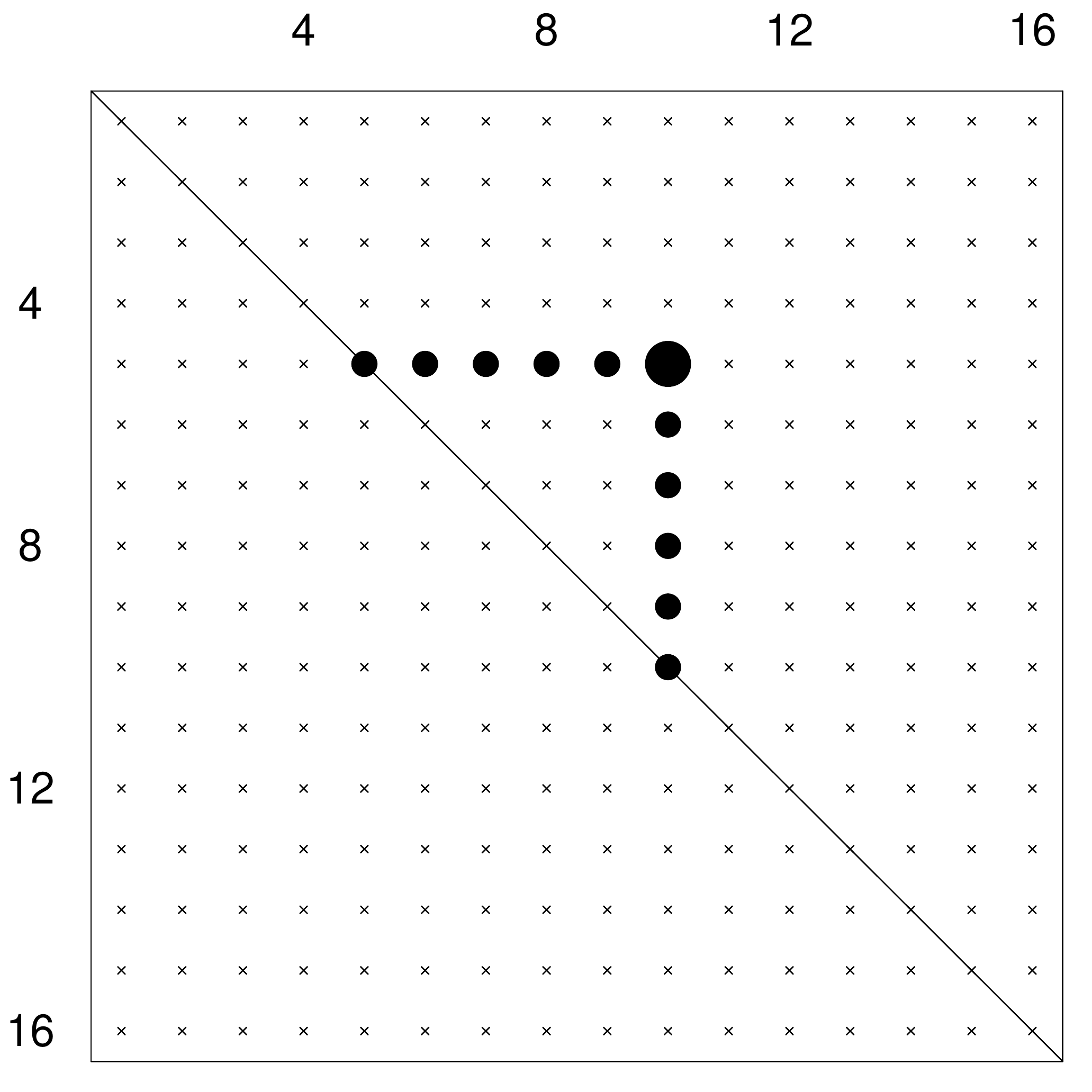}
  \includegraphics[width=0.25\textwidth]{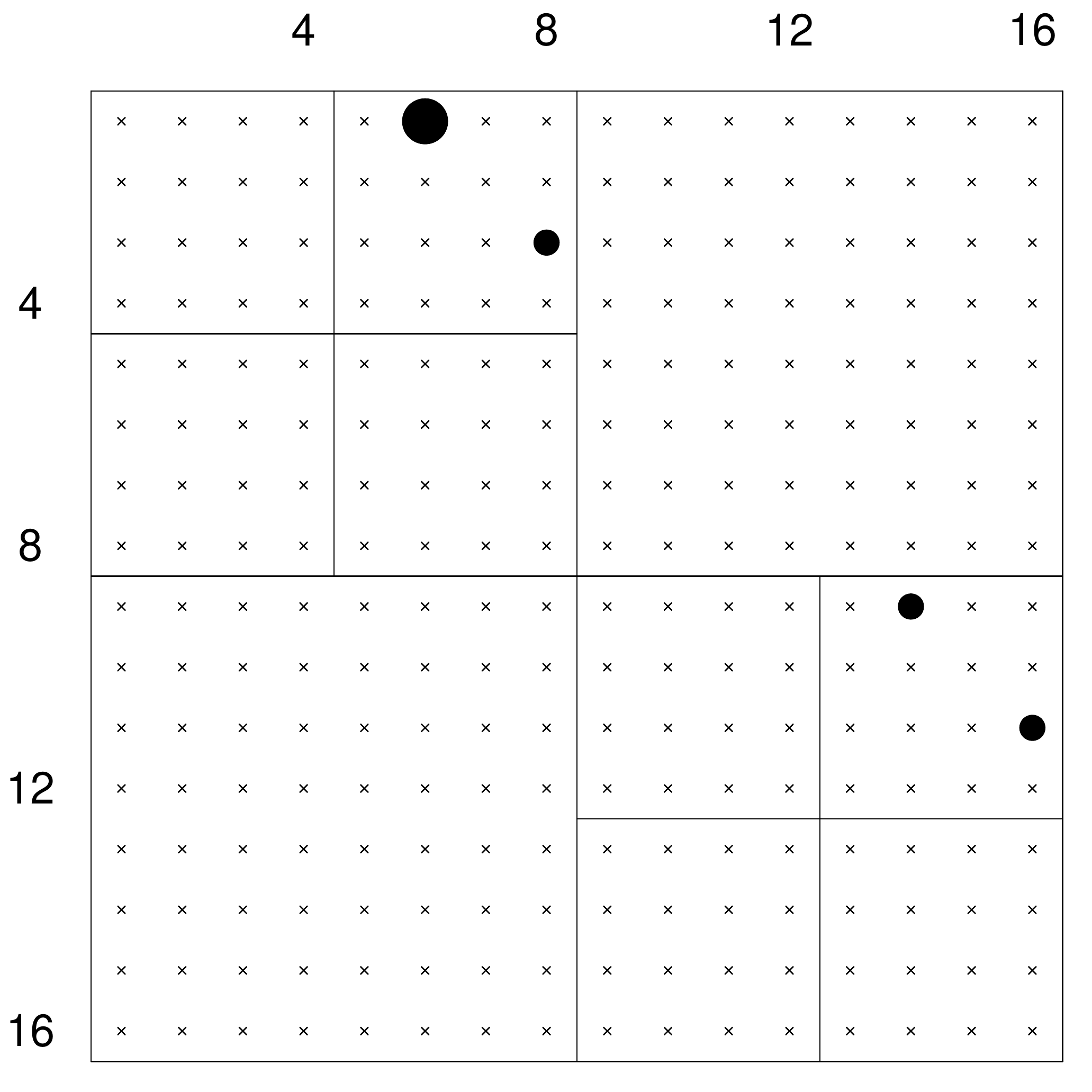}
  \includegraphics[width=0.25\textwidth]{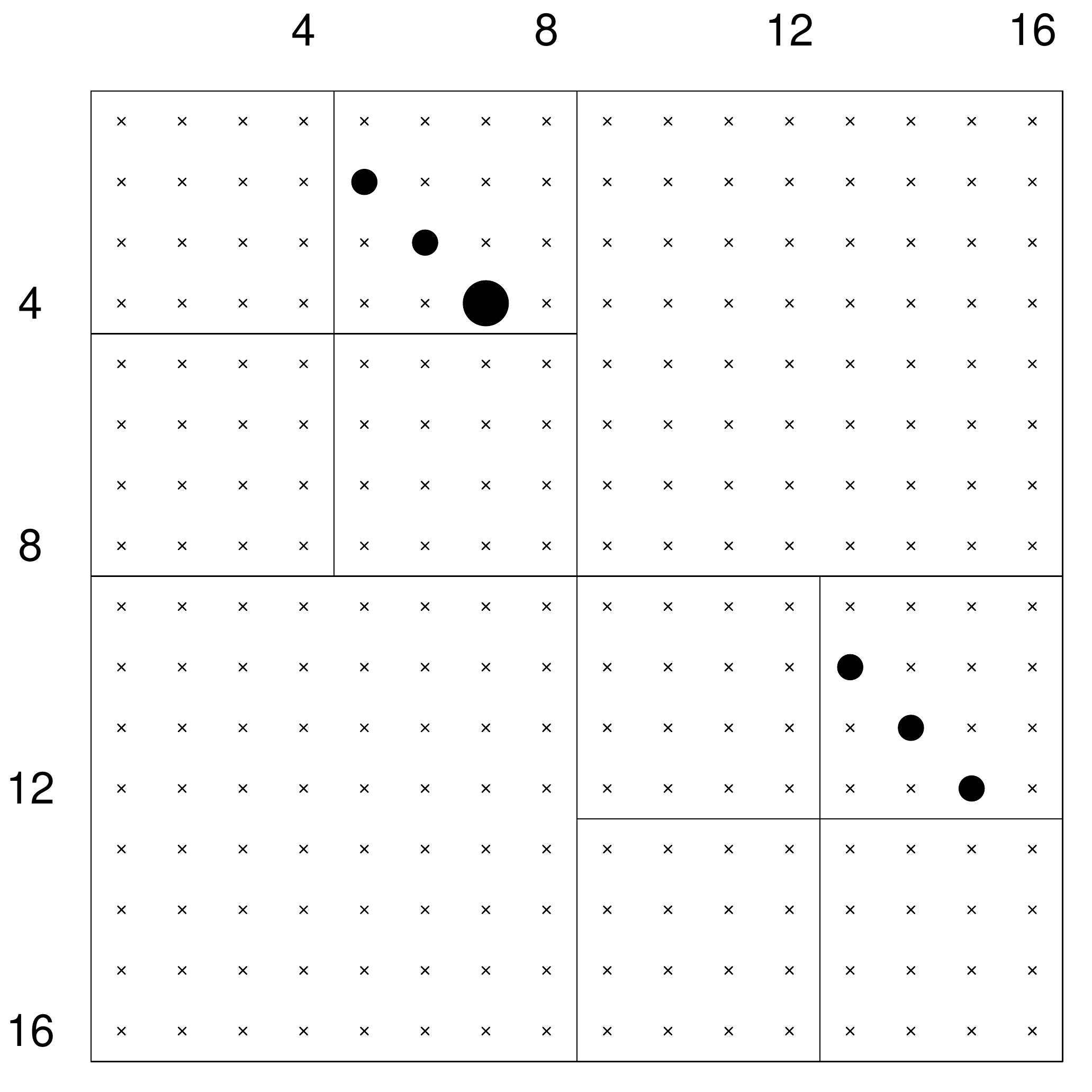}
	\captionsetup{singlelinecheck=off}
	\caption[Patterns in distance matrices]{Bigger white dots, respectively black ones, represent the initial value for patterns, so do smaller dots as consequent values. From the left to the right, we display:
	\begin{itemize}
		\item patterns $A_0$, $B_0$, $C_0$ in the first row,
		\item patterns $A_1$, $B_1$, $C_1$ in the second row.
	\end{itemize}
	}
	\label{fig:patterny}
\end{figure}

The next lemma shows that a distance matrix contains all the patterns. Their importance is accentuated because of radical simplifying an estimate of the ratio of $\ve$--close pairs of intervals to all interval pairs.

\begin{lema}
\label{lema:postupnost_gulky}
The distance matrix $D_k(\ve)$ for each $k > 0$ contains the patterns $A_0$, $A_1$, $B_0$, $B_1$, $C_0$, $C_1$.
\end{lema}
\begin{proof}
The patterns $A_0$ and $A_1$ follow from the definition of distance matrix and the geometry of $I(u)$ with $u\in\Sigma^k$ on $[0, 1]$. Choose $u\in\Sigma^k$ and let $s = s_{\gamma(u)}$, $1\leq m < k$ be as in the definition of patterns. For the patterns $C_0$, $C_1$, assume that $\gamma(u)\geq 2^m$. Then:
\begin{itemize}
    \item[($B_0$)] $A(n + 1, \gamma(u) - 1) \geq A(1, \gamma(u) - 1)$. It is always true by Lemma~\ref{lema:postupnost_najmensie}.
    \item[($B_1$)] $A(n 2^{s + 1} + 1, \gamma(u) - 1) = A((n + 1) 2^{s + 1} + 1 - \gamma(u), \gamma(u) - 1) = A(1, \gamma(u) - 1)$. That follows from~(\ref{lema:a_s}) since $\gamma(u) - 1 < 2^{s + 1}$.
    \item[($C_0$)] $A(2^m + h 2^{s + 1}, \gamma(u) - 2^m) = A((h + 1) 2^{s + 1} - \gamma(u) + 1, \gamma(u) - 2^m) = A(2^m, \gamma(u) - 2^m)$. This case is analogous to the previous one.
    \item[($C_1$)] $A(2^m + h 2^{s + 1} - n, \gamma(u) - 2^m) \leq A(2^m, \gamma(u) - 2^m)$. Since $2^m\leq 2^s < \gamma(u) \leq 2^{s + 1}$ and $n < 2^m$, the inequality follows from Lemma~\ref{lema:postupnost_najvacsie} where $n' = \gamma(u)$, $s' = s + 1$ and $h' = n$.
    \end{itemize}

\noindent
From~(\ref{eq:gamma_scitanie}), we have $\gamma(u\pp n) - \gamma(v\pp n) = \gamma(u) - \gamma(v)$. Therefore
\begin{align*}
\gamma(0^{k - s - 1}1^{s + 1} \pp& n\cdot 2^{s + 1}) - \gamma(0^{k - s - 2}10^{s + 1}\pminus \gamma(u)\pp n\cdot 2^{s + 1}) = \\
&= 2^{s + 1} + n\cdot 2^{s + 1} - (2^{s + 1} + 1 - \gamma(u) + n\cdot 2^{s + 1}) = \gamma(u) - 1 = \\
&= \gamma(u) - \gamma(0^k),
\end{align*}
and
\begin{align*}
\gamma(0^{k - s - 1}&1^{s + 1}\pminus 2^m \pp 1 \pp h\cdot 2^{s + 1}) - \gamma(0^{k - s - 1}1^{s + 1}\pminus \gamma(u)\pp 1\pp h\cdot 2^{s + 1}) = \\
&= 2^{s + 1} - 2^m + 1 + h\cdot 2^{s + 1} - (2^{s + 1} - \gamma(u) + 1 + h\cdot 2^{s + 1}) = \gamma(u) - 2^m = \\
&= \gamma(u) - \gamma(0^{k - m}1^m).
\end{align*}

Using definitions of patterns together with~(\ref{eq:gap_seq_dis}) we proved the lemma.
\end{proof}

\section{Correlation integrals and asymptotic determinisms of maps \texorpdfstring{$f_\alpha$}{fa} and \texorpdfstring{$f_{\alpha, k}$}{fak}}
\label{del:cor_int}
\label{kap_cor_int}

Let $1\leq \ell\leq\infty$ and $k > 0$ be given. Then $\sup\{|f_{k'}(x) - f(x)|\ |\ x\in [0, 1]\} \leq \alpha^k$ for all $k'\geq k$. By Definition~\ref{df:del} and the definition of intervals~(\ref{eq:interval}), $f^i(I(u)) = f_k^i(I(u)) = f_{k'}^i(I(u))$ for every $u\in\Sigma^k$, $i\geq 0$, and $k'\geq k$. Fix $x\in X_k$, then $\dis(f^i(x), f_{k'}^i(x))\leq\alpha^k$ for all $k'\geq k$ and $i \geq 0$. If $\dis(f^i(x), f^j(x))\leq \ve$ where $i, j\geq 0$, then we have
\begin{align*}
\dis(f_{k'}^i(x), f_{k'}^j(x)) \leq \dis(f_{k'}^i(x), f^i(x)) + \dis(f^i(x), f^j(x)) + \dis(f^j(x), f_{k'}^j(x))\leq \ve + 2\cdot\alpha^k
\end{align*}
for every $k'\geq k$. Similarly, $\dis(f_{k'}^i(x), f_{k'}^j(x))\leq \ve - 2\cdot\alpha^k$ implies $\dis(f^i(x), f^j(x))\leq \ve$. From the definition of $\dis_\ell$,
\begin{align*}
& \dis_{\ell}^{f_{k'}}(f_{k'}^i(x), f_{k'}^j(x)) \leq \ve - 2\cdot\alpha^k \Rightarrow \dis_{\ell}^f(f^i(x), f^j(x)) \leq \ve\\
& \dis_{\ell}^f(f^i(x), f^j(x)) \leq \ve \Rightarrow \dis_{\ell}^{f_{k'}}(f_{k'}^i(x), f_{k'}^j(x)) > \ve + 2\cdot\alpha^k.
\end{align*}

Then, by~(\ref{eq:cor_sum}), for all $n > 0$, $k'\geq k$, and $x\in X_k$,
\begin{equation}
\label{eq:dvaja_policajti}
\tC{\ell}{f_{k'}}(x, n, \ve - 2\cdot\alpha^k) \leq \tC{\ell}{f}(x, n, \ve) \leq \tC{\ell}{f_{k'}}(x, n, \ve + 2\cdot\alpha^k).
\end{equation}

\begin{lema}
\label{lema:periodic}
Let $(M, \dis)$ be a metric space and $g$ a map on $M$. Suppose that $x$ is a periodic point with the period $p$. Then $\lim_{n\to\infty}\tC{\ell}{g}(x, n, \ve)$ exists and is equal to $\tC{\ell}{g}(x, p, \ve) = \mu_x\times\mu_x\{(x, y)\in M^2\ |\ \dis_{\ell}(x, y)\leq\ve\}$ where $\mu_x$ is a uniform measure with the support on the orbit of $x$.
\end{lema}
\begin{proof}
By assumptions, $\dis_{\ell}(g^i(x), g^j(x)) = \dis_{\ell}(g^{i + m_1\cdot p}(x), g^{j + m_2\cdot p}(x))$ for every $i, j \geq 0$ and $m_1, m_2 \geq 0$. Every $n\in\NNN$ can be uniquely written as $n = m_n\cdot p + q_n$ where $m_n \geq 0$ and $0\leq q_n < p$. It follows that
\begin{align*}
\card\{(i, j)\                  &|\   0\leq i, j < n, \dis_{\ell}(g^i(x), g^j(x))\leq \ve\} = \\
                                & = m_n^2\card\{(i, j)\ |\ 0\leq i, j < p, \dis_{\ell}(g^i(x), g^j(x))\leq \ve\} + \\
                                & + \card\{(i, j)\ |\ 0\leq i, j < n, \max\{i, j\} > m_n\cdot p, \dis_{\ell}(g^i(x), g^j(x))\leq \ve\}.
\end{align*}
Thus
\begin{align*}
  (m_np)^2\cdot \tC{\ell}{g}(x, p, \ve) \leq (m_n\cdot p + q_n)^2 \cdot \tC{\ell}{g}(x, n, \ve) \leq (m_np)^2\cdot \tC{\ell}{g}(x, p, \ve) + 2nq_n - q_n^2.
\end{align*}
The first statement follows after multiplying the inequality by $1/n^2$.

Since $\mu_x$ is uniform, and non-zero only on the orbit of $x$, we have
\begin{align*}
\mu_x\times\mu_x\{(x, y)\in M^2\ |\ \dis_{\ell}(x, y) \leq\ve\} = \frac{1}{p^2}\card\{(x, y)\in\Orb_g^2(x)\ |\ \dis_{\ell}(x, y)\leq\ve\}
\end{align*}
which is the same as $\tC{\ell}{g}(x, p, \ve)$.
\end{proof}

Each $f_{k}$ is periodic on $X_k$, hence $\lim_{n\to\infty} \tC{\ell}{f_{k}}(x, n, \ve)$ exists for every $x\in X_k$ and is equal to $\tC{\ell}{f_{k}}(x, 2^k, \ve)$. Left points of intervals $I(u)$, here being $u\in\Sigma^k$, form the orbit $\Orb_{f_k}(0)$. Using the distance matrix $D_{k, \ell}(\ve)$ by the previous lemma we can compute the $\ell$--correlation sum for $x = 0$.

\begin{lema}
\label{lema:per_det_one}
Let $g$ be a continuous map on a metric space $(M, \dis)$. Let $x\in M$ be periodic under $g$. Then there is $\ve_0 > 0$ such that for every $\ve < \ve_0$, $n > 0$ and every $1\leq\ell\leq\infty$ the determinism $\tDET{\ell}{g}(x, n, \ve)$ equals $1$.
\end{lema}
\begin{proof}
Suppose that $p$ is the period of $x$. Set $\ve_0 = \min\{\dis(f^i(x), f^j(x))\ |\ 0\leq i, j < p\}$. Thus for every $0 < \ve < \ve_0$, if $\dis_{\ell}(f^i(x), f^j(x))\leq\ve$, we have $j = q\cdot p + i$ for some $q\geq 0$.

The statement follows from the definition of determinism.
\end{proof}

Let $(M, \dis)$ be a metric space and $g: M\to M$ be a continuous map. Recall that we say that a measure $\nu$ is \emph{$g$--ergodic} if every subset $A\subset M$ with the property $g^{-1}(A) = A$ satisfies $\nu(A) = 0$ or $\nu(M\setminus A) = 0$. In addition, let $\nu$ be a probability measure, that is $\nu(M) = 1$.

The \emph{$\ell$--correlation integral} is defined by
\begin{equation}
\label{def:kor-int}
\tc{\ell}{g}(\nu, \ve) = \nu\times\nu\{(x, y)\in M\times M\ |\ \dis_{\ell}(x, y) \leq \ve\}
\end{equation}
and the \emph{$\ell$--asymptotic determinism} is defined by
\begin{equation}
\label{def:as-det}
\tdet{\ell}{g}(\nu, \ve) = \frac{\ell\cdot \tc{\ell}{g}(\nu, \ve) - (\ell - 1)\cdot \tc{\ell + 1}{g}(\nu, \ve)}{\tc{1}{g}(\nu, \ve)}\qquad\text{and especially}\qquad
\tdet{\infty}{g}(\nu, \ve) = \frac{\tc{\infty}{g}(\nu, \ve)}{\tc{1}{g}(\nu, \ve)}.
\end{equation}
The $\ell$--correlation integral has originally been defined in~\cite{GP-cor-sum} as the limit of $(\tC{\ell}{g}(x, i, \ve))_i$ where $x\in M$. Definitions above follow from the results in~\cite{pesin} and~\cite{pesin-tempelman} and from Lemma~\ref{lema:korelacny_rekurent}. Later, ergodic measures will be fixed, therefore we can leave them from notation and write $\tc{\ell}{g}(\ve)$ instead of $\tc{\ell}{g}(\nu, \ve)$. Similarly for determinisms.

Define
\begin{align*}
\mu_{\alpha, k}(A) = 1/2^k\cdot\card\{\Orb_{f_{\alpha, k}}(0)\cap A\}
\end{align*}
for a subset $A\subset [0, 1]$. The set function $\mu_{\alpha, k}$ is a measure. Similarly, set $\mu_{\alpha}$ to be a unique ergodic measure on $[0, 1]$ with the support on the set $X_{\alpha}$, i.e.~$\mu_{\alpha}(I_{\alpha}(u^{(k)})) = 1/2^k$ for all $k > 0$ and $u^{(k)}\in\Sigma^k$, and the measure of a measurable $B\subset [0, 1]\setminus X_{\alpha}$ is zero.

\begin{lema}
\label{lema:hranice_rr}
Let $k \geq 1$, $h\geq 0$ and $\ve\in(\alpha^{h + 1}, \alpha^h]$ be given. Suppose that $u_m\in\Sigma^{k + h}$ where $1\leq m\leq h + 1$ satisfy $2^{k - 1} < \gamma(u_1) \leq 2^k$ and $2^{k + m - 2} \leq \gamma(u_m) < 2^{k + m - 1}$ for $m > 1$. Set $j_1 = 2^k - \gamma(u_1)$ and $j_m = \gamma(u_m) - 2^{k + m - 2}$ for $2\leq m \leq h + 1$.
Assume that the pair of $I(0^{k + h})$ and $I(u_1)$ is $\ve$--close, and $I(0^{k + h})$ and $I(u_1 \pp 1)$ are $\ve$--distant. Similarly, require that $I(0^{h - m + 2}1^{k + m - 2})$ and $I(u_{m})$ are $\ve$--close, and $I(0^{h - m + 2}1^{k + m - 2})$ and $I(u_{m} \pp 1)$ are $\ve$--distant. Then
\begin{align*}
\left(\frac{1}{2^{k + h}}\right)^2 \left[2^{h + 1}(2^{2k - 1} - j_1^2) + \sum_{m = 2}^{h + 1}2^{h - m + 1}j_{m}^2\right]&\leq \tc{1}{f_{k}}(\ve) \leq\\
                                                                            \leq \left(\frac{1}{2^{k + h}}\right)^2 &\left[2^h(2^{2k} - j_1^2) + \sum_{m = 2}^{h + 1}2^{h - m + 2}j_{m}^2\right].
\end{align*}
\end{lema}

\begin{figure}[!ht]
    \centering
	\includegraphics[width = 0.3\textwidth]{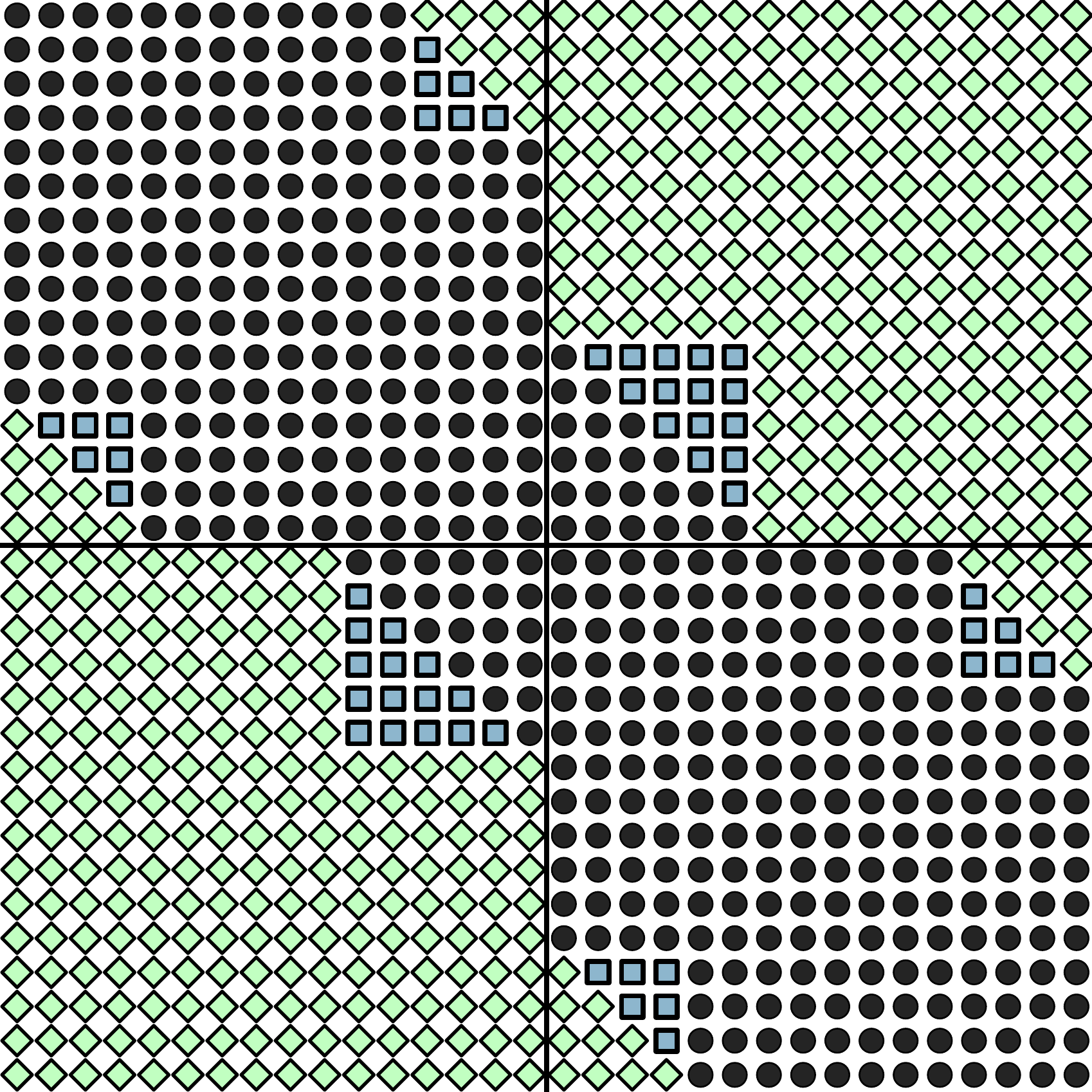}
	\caption[Example of a distance matrix]{Lemma~\ref{lema:hranice_rr} for $k = 4$ and $h = 1$. Here $u_1 = 01011$ and $u_2 = 10101$. Black dots are pairs of intervals which are $\ve$--close. Green diamonds are pairs of intervals which are $\ve$--distant. Blue squares are pairs of intervals which may or may not be $\ve$--close.}
	\label{fig:vysvetlenie_veta}
\end{figure}

\begin{proof}
We can fill almost the whole upper triangle of distance matrix by using patterns $A_0$, $A_1$, $B_0$, $B_1$, $C_0$, $C_1$ (Lemma~\ref{lema:postupnost_gulky}). Since a metric is commutative, $D_{k + h}(\ve)[i, j] = D_{k + h}[j, i]$. We can imagine all matrices being similar to the one in Figure~\ref{fig:vysvetlenie_veta} where 
\begin{itemize}
	\item black dots are given by patterns $A_1$, $B_1$, $C_1$, 
	\item green diamonds are given by patterns $A_0$, $B_0$, $C_0$ and 
	\item blue squares are in positions in the distance matrix which are not determined by the patterns. 
\end{itemize}
Then, the correlation integral is at least $1/2^{2(k + h)}\times$ ''the number of black dots in such a figure'' and is less than or equal to $1/2^{2(k + h)}\times$ ''the number of black dots plus blue squares''.
\end{proof}

Note that the length of the above-mentioned interval covering $\tc{\ell}{f_k}(\ve)$ need not converge to zero with $k\to\infty$. However, there are special cases for which the interval is degenerate.

Obviously, the correlation integral as a function of radius $\ve$ is non-decreasing. The next lemma shows a similar relation between correlation integrals for the map $f$ and its approximations $f_k$.

\begin{lema}
\label{lema:c_funkcia}
Let $g:\NNN\to [0, 1]$ be a function such that $g(k)\geq \alpha^k$. For all $1\leq \ell \leq \infty$ and large enough $k$,
\begin{align*}
\tc{\ell}{f_k}(\ve - g(k))\leq \tc{\ell}{f}(\ve)\leq \tc{\ell}{f_k}(\ve + g(k)).
\end{align*}
\end{lema}
\begin{proof}
We have to prove that
\begin{align*}
\mu_k\times\mu_k\{(x, y)\ |\ \dis_{\ell}^{f_k}(x, y) \leq \ve - g(k)\}  &\leq \mu\times\mu\{(x, y)\ |\ \dis_{\ell}^f(x, y) \leq \ve\} \leq \\
                                                                        &\leq \mu_k\times\mu_k\{(x, y)\ |\ \dis_{\ell}^{f_k}(x, y) \leq \ve + g(k)\}.
\end{align*}

Denote the left point of $I(u)$ by $x_u$ and the right one by $y_u$. By definitions of measures, $\mu_k\times\mu_k((x_u, x_v)) = \mu\times\mu(I(u)\times I(v))$ and $\mu_k$ is zero on the set $X_0\setminus\{(x_u, x_v)\ |\ u, v\in\Sigma^k\}$. It is sufficient to show that
\begin{align*}
& (x_u, x_v)\in \{(x, y)\ |\ \dis^{f_k}_{\ell}(x, y) \leq \ve - \alpha^k\} \Rightarrow I(u)\times I(v) \subset \{(x, y)\ |\ \dis^f_{\ell}(x, y) \leq \ve\},\\
& I(u)\times I(v) \cap \{(x, y)\ |\ \dis^f_{\ell}(x, y) \leq \ve\}\neq \emptyset \Rightarrow (x_u, x_v)\in \{(x, y)\ |\ \dis^{f_k}_{\ell}(x, y) \leq \ve + \alpha^k\}.
\end{align*}
The biggest distance of points from the intervals $I(u)$ and $I(v)$, where $u \leq v$, is for $x_u$ and $y_v$. Since the length of an interval is $\alpha^k$, we have $\dis(x_u, y_v) = \dis(x_u, x_v) + \alpha^k$. Similarly, points with the smallest distance are represented by the right lower corner of $I(u)\times I(v)$, i.e.~$(y_u, x_v)$. Then $\dis(x_u, x_v) = \dis(y_u, x_v) + \alpha^k$. If $(x, y)\in I(u)\times I(v)$, we thus have
\begin{equation}
\label{eq:stvorce}
\dis(x_u, x_v) - \alpha^k\leq \dis(x, y)\leq \dis(x_u, x_v) + \alpha^k.
\end{equation}

Suppose that $\sup\{\dis(f_k^i(x_u), f_k^i(x_v))\ |\ 0\leq i < \ell\} \leq \ve - \alpha^k$. Since $f_k\times f_k(x, y)$ and $f\times f(x, y)$ belong to the same square and $f_k\times f_k(x_u, x_v) = (x_{u\lp 1}, x_{v\lp 1})$, from~(\ref{eq:stvorce}) we have that $\sup\{\dis(f^i(x), f^i(y))\ | \ 0\leq i < \ell\} \leq \ve$ for every $(x, y)\in I(u)\times I(v)$.

Conversely, suppose that $\sup\{\dis(f^i(x), f^i(y))\ |\ 0\leq i < \ell\} \leq \ve$ for some $(x, y)\in I(u)\times I(v)$. Thus similarly as before $\sup\{\dis(f_k^i(x_u), f_k^i(x_v))\ |\ 0\leq i < \ell\} \leq \ve + \alpha^k$.

Since $g(k)\geq \alpha^k$ for every $k$ and $\tc{\ell}{f_k}$ is a non-decreasing function, the proof is complete.
\end{proof}

Denote the smallest integer not less than $x\in\RRR$ by $\left\lceil x\right\rceil$. Let $J\subset [0, 1]$ be a subinterval of length $\beta$. Then there are at most $h_{J} = \left\lceil \beta/\alpha^k\right\rceil + 1$ words $u\in \Sigma^k$ such that $I(u)\cap J\neq\emptyset$. Let $h_{\beta}$ be the maximal number of words $u\in\Sigma^k$ such that $I(u)\cap J'\neq \emptyset$ where $J'\subset [0, 1]$ is an interval of length $\beta$. Let $u_1 < u_2 < \ldots < u_{h_{\beta}}$ be these words. Clearly, $h_{\beta}\leq h_J$ and $u_i = u_1\pp (i - 1)$.

\begin{lema}
\label{lema:f_k_prechod_f}
Let $g:\NNN\to [0, 1]$ satisfy $g(k)\geq \alpha^k$. If $\lim_{i\to\infty} h_{g(i)}/2^i$ exists and is equal to zero, then $\tc{\ell}{f_k}(\ve + g(k)) - \tc{\ell}{f_k}(\ve - g(k))\to 0$ as $k\to\infty$.
\end{lema}
\begin{proof}
In the first row of $D_k(\ve + g(k))$ there are at most $2\cdot h_{g(k)}$ ones that are not in the first row of $D_k(\ve - g(k))$. Therefore, from Corollary~\ref{dosledok:pocet_blizke_orbity}, there are at most $2^k\cdot 2\cdot h_{g(k)}$ pairs of words $u, v\in\Sigma^k$ such that $\ve - g(k) < \dis_{\infty}(u, v)\leq \ve + g(k)$. Thus
\begin{align*}
\tc{\infty}{f_{k}}(\ve + g(k)) - \tc{\infty}{f_{k}}(\ve - g(k)) \leq \frac{2\cdot h_{g(k)}}{2^k}.
\end{align*}
Analogously,
\begin{align*}
\tc{1}{f_{k}}(\ve + g(k)) - \tc{1}{f_{k}}(\ve - g(k)) \leq \frac{4\cdot h_{g(k)}}{2^k}.
\end{align*}

Now, let $\ell < \infty$. We prove this case from more geometrical point of view motivated by~\cite{Manning}.

Consider a picture of the square $[0, 1]^2$ together with $2^{2k}$ smaller squares $I(u)\times I(v)$ where $u, v\in\Sigma^k$. For every $u, v\in\Sigma^k$, the measures $\mu\times\mu(I(u)\times I(v)) = 1/2^{2k} = \mu_k\times\mu_k(I(u)\times I(v))$. Let the lines $y = x + \ve$ and $y = x - \ve$ be drawn in the picture. Each line intersects at most two squares in every column, that is $\mu_k\times\mu_k\{(x, y)\in X\times X\ |\ \dis_{1}^{f_k}(x, y) = \ve\} \leq 4/2^k$. The function $f_k$ maps the picture on the same one and therefore $\mu_k\times\mu_k\{(x, y)\in X\times X\ |\ \dis_{1}^{f_k}(f_k(x), f_k(y)) = \ve\} \leq 4/2^k$. This gives
\begin{align*}
\mu_k\times\mu_k\{(x, y)\in X^2\  &|\ \dis_{2}^{f_k}(x, y) = \ve\} \leq \mu_k\times\mu_k\{(x, y)\in X^2\ |\ \dis_{1}^{f_k}(x, y) = \ve\} + \\
										& + \mu_k\times\mu_k\{(x, y)\in X^2\ |\ \dis_{1}^{f_k}(f_k(x), f_k(y)) = \ve\} \leq 2\cdot \frac{4}{2^k}.
\end{align*}
Similarly, $\mu_k\times\mu_k\{(x, y)\in X^2\ |\ \dis_{\ell}^{f_k}(x, y) = \ve\} \leq \ell\cdot 4/2^k$.

By the definition of correlation integrals,
\begin{align*}
\tc{\ell}{f_k}(\ve + g(k)) - \tc{\ell}{f_k}(\ve - g(k)) = \mu_k\times\mu_k\{(x, y)\in X^2\ |\ \ve - g(k) < \dis_{\ell}^{f_k}(x, y)\leq\ve + g(k)\}.
\end{align*}

We again consider the picture as before, but instead of lines $y = x \pm \ve$, we do consider lines $y = x + \ve \pm g(k)$ and $y = x - \ve \pm g(k)$. Let $k$ be large enough that $\ve - g(k) > g(k) - \ve$. The number of squares with points lying between the lines $y = x + \ve + g(k)$ and $y = x + \ve - g(k)$, respectively between $y = x - \ve + g(k)$ and $y = x - \ve - g(k)$, is at most $8\cdot h_{g(k)}$. Therefore, similarly to the above-mentioned case,
\begin{align*}
\mu_k\times\mu_k\{(x, y)\in X\times X\ |\ \ve - g(k) < \dis_{\ell}^{f_k}(x, y)\leq\ve + g(k)\} \leq 8\cdot h_{g(k)}\cdot\ell/2^k.
\end{align*}
Now, the statement is proved.
\end{proof}

From previous lemmas, we can use the correlation integral $\tc{\ell}{f_k}$ for approximation of the correlation integral $\tc{\ell}{f}$. The choice of $k$ depends on neither $\ve$ nor $\alpha$ which is clear from the previous proof. However, the choice of $k > 0$ satisfying $|\tdet{\infty}{f}(\ve) - \tdet{\infty}{f_k}(\ve)| < \varepsilon$ does depend on $\ve$.

\begin{dosledok}
\label{cor:rozdiel_ck_c}
For every $k \geq 1$ such that $\ve > \alpha^{k}$, the maximal errors of approximation are
\begin{eqnarray*}
\begin{array}{ll}
|\tc{1}{f_k}(\ve) - \tc{1}{f}(\ve)|            	&< 8/2^k\\
|\tc{\ell}{f_k}(\ve) - \tc{\ell}{f}(\ve)|    		&< 16\cdot\ell/2^k\\
|\tc{\infty}{f_k}(\ve) - \tc{\infty}{f}(\ve)|  	&< 4/2^k,
\end{array}
\end{eqnarray*}
where $1 < \ell < \infty$. If $\ve\in(\alpha^{h + 1}, \alpha^h]$ where $h\geq 0$, then
\begin{align*}
\left|\tdet{\infty}{f_{\alpha}}(\ve) - \tdet{\infty}{f_{\alpha, k}}(\ve)\right| \leq \frac{24}{2^{k - h - 1} - 8}.
\end{align*}
\end{dosledok}
\begin{proof}
From Lemma~\ref{lema:c_funkcia}, the integral $\tc{\ell}{f}(\ve) \in (\tc{\ell}{f_k}(\ve - \alpha^k), \tc{\ell}{f_k}(\ve + \alpha^k))$. Clearly, $\tc{\ell}{f_k}(\ve)$ belongs into this interval, thus $|\tc{\ell}{f_k}(\ve) - \tc{\ell}{f}(\ve)| < \tc{\ell}{f_k}(\ve + \alpha^k) - \tc{\ell}{f_k}(\ve - \alpha^k)$. The length of each $I(u)$ where $u\in\Sigma^k$ is $\alpha^k$, therefore $h_{\alpha^k} = 2$. The statement about correlation integrals now follows from Proof of Lemma~\ref{lema:f_k_prechod_f}.

By an easy computation, 
\begin{itemize}
	\item $32/(2^k \tc{1}{f_k}(\ve)(2^k \tc{1}{f_k}(\ve) - 8)) = -4/(2^k \tc{1}{f_k}(\ve)) + 4/(2^k \tc{1}{f_k}(\ve) - 8)$ and similarly, 
	\item $32/(2^k \tc{1}{f_k}(\ve)(2^k \tc{1}{f_k}(\ve) + 8)) = 4/(2^k \tc{1}{f_k}(\ve)) - 4/(2^k \tc{1}{f_k}(\ve) - 8)$. 
\end{itemize}	
Therefore using the worst approximations of correlation integrals, we have
\begin{align*}
\tdet{\infty}{f_k}(\ve) - \frac{12}{2^k \tc{1}{f_k}(\ve) + 8} &= \tdet{\infty}{f_k}(\ve) - \frac{8}{2^k \tc{1}{f_k}(\ve) + 8} - \frac{4}{2^k \tc{1}{f_k}(\ve)} + \frac{32}{2^k \tc{1}{f_k}(\ve)\cdot(2^k \tc{1}{f_k}(\ve) + 8)}\leq \\
    &\leq \frac{\tc{\infty}{f_k}(\ve) - 4/2^k}{\tc{1}{f_k} + 8/2^k} \leq \tdet{\infty}{f}(\ve) \leq \frac{\tc{\infty}{f_k}(\ve) + 4/2^k}{\tc{1}{f_k} - 8/2^k}\leq \\
    \leq \tdet{\infty}{f_k}(\ve) + \frac{8}{2^k \tc{1}{f_k}(\ve) - 8} &+ \frac{4}{2^k \tc{1}{f_k}(\ve)} + \frac{32}{2^k \tc{1}{f_k}(\ve)\cdot(2^k \tc{1}{f_k}(\ve) - 8)} = \tdet{\infty}{f_k}(\ve) + \frac{12}{2^k \tc{1}{f_k}(\ve) - 8}.
\end{align*}
Thus $|\tdet{\infty}{f}(\ve) - \tdet{\infty}{f_k}(\ve)| \leq 24/(2^k \tc{1}{f_k}(\ve) - 8)$. The integral $\tc{1}{f_k}(\ve) \geq 2^{- h - 1}$ for $\alpha^{h + 1} < \ve \leq \alpha^h$ which follows from $\dis(0^k, 0^h 1^{k - h}) < \alpha^{h + 1}$  and from $A_1$, $B_1$.
\end{proof}

By the previous corollary, we can use the distance matrix not only for $f_k$ but for $f$ as well. Next theorem shows that the matrix with a small change of $\ve$ cannot change much.

\begin{thm}
\label{thm:prechod}
For every $1\leq\ell\leq\infty$, functions $\tc{\ell}{f}$, $\tdet{\ell}{f}$ are continuous at $\ve$, and $\tc{\ell}{f}(\ve) = \lim_{k\to\infty}\tc{\ell}{f_k}(\ve)$ and $\tdet{\ell}{f}(\ve) = \lim_{k\to\infty}\tdet{\ell}{f_{k}}(\ve)$.
\end{thm}
\begin{proof}
The latter statements follow immediately from Corollary~\ref{cor:rozdiel_ck_c}.

Put $k_i = k_0 + i - 1$ with $k_0 > 0$ such that $\ve > \alpha^{k_0}$. Set $\ve_{2i} = \ve + \alpha^{k_i}$ and $\ve_{2i - 1} = \ve - \alpha^{k_i}$. Then $\ve_i\to\ve$ as $i\to\infty$, and for all $i\in\NNN$ we have that $\ve_{2i - 1} < \ve < \ve_{2i}$. From Lemma~\ref{lema:f_k_prechod_f}, the difference $\tc{\ell}{f_{k_i}}(\ve_{2i}) - \tc{\ell}{f_{k_i}}(\ve_{2i - 1})\to 0$ as $i\to\infty$ and from Corollary~\ref{cor:rozdiel_ck_c}, $\tc{\ell}{f}(\ve') - \tc{\ell}{f_{k'}}(\ve')\to 0$ as $k'\to\infty$ for all $\ve' > 0$. Thus $\tc{\ell}{f}(\ve_{2i}) - \tc{\ell}{f}(\ve_{2i - 1})\to 0$ as $i\to\infty$ and the monotonicity of $\tc{\ell}{f}$ completes the proof.
\end{proof}

Denote the set of all limit points of $(\tC{\ell}{f}(x, i, \ve))_i$ by $\tc{\ell}{f}(x, \ve)$ and the set of limit points of $(\tDET{\ell}{f}(x, i, \ve))_i$ by $\tdet{\ell}{f}(x, \ve)$. If $\tc{\ell}{f}(x, \ve)$, respectively~$\tdet{\ell}{f}(x, \ve)$, is a singleton, we identify the set with the value of its unique element. Since $\tc{\ell}{f}$ is continuous at every $\ve$, we have $\tc{\ell}{f}(x, \ve) = \tc{\ell}{f}(\ve)$ for $\mu$--a.e.~$x\in [0, 1]$, cf.~\cite{pesin-tempelman}. We show that the limit of $(\tC{\ell}{f}(x, i, \ve))_i$ exists for all $x\in[0, 1]$. First, we will prove the convergence for $x = 0$.

\begin{lema}
\label{lema:c_l_limita}
Let $1\leq\ell\leq\infty$. Then $\tc{\ell}{f}(0, \ve) = \tc{\ell}{f}(\ve)$.
\end{lema}
\begin{proof}
We can use bounds~(\ref{eq:dvaja_policajti}) for every $k\geq 1$ since $0$ belongs to every $X_k$. Thus
\begin{align*}
\tC{\ell}{f_{k}}(0, n, \ve - 2\cdot\alpha^k) \leq \tC{\ell}{f}(0, n, \ve) \leq \tC{\ell}{f_{k}}(0, n, \ve + 2\cdot\alpha^k).
\end{align*}
The point $0\in X$ is periodic for every $f_k$, then from Lemma~\ref{lema:periodic} and the definition of $\tc{\ell}{f_k}(\ve)$,
\begin{align*}
\tc{\ell}{f}(0, \ve)\subset \bigcap_{k\geq 0} \left[\tc{\ell}{f_k}(\ve - 2\cdot\alpha^k), \tc{\ell}{f_k}(\ve + 2\cdot\alpha^k)\right].
\end{align*}
Using Lemma~\ref{lema:f_k_prechod_f}, where $g(k) = 2\alpha^k$ and $h_{g(k)} = 4$, $\tc{\ell}{f}(0, \ve)$ is a singleton. By Theorem~\ref{thm:prechod}, its value is $\tc{\ell}{f}(\ve)$.
\end{proof}

Denote by $\eper(f)$ the set of \emph{eventually periodic points} of the map $f$, i.e.~if $x\in\eper(f)$, then there are $i\geq 0$ and $p > 0$ with the property that $f^{m\cdot p}(f^i(x)) = f^i(x)$ for every $m \geq 0$. We call such smallest $p > 0$ the period of $x$. 

Let $x\in [0, 1]$ be a $2^k$--periodic point under $f_k$, that is $x\in X_k$. Since 
\begin{itemize}
	\item $\restr{f_k}{I(u)}$ is linear and its slope is $1$ and
	\item $f_k^i(I(u)) = I(u \lp i)$
\end{itemize}
for each $u\in\Sigma^k$ and $i \geq 0$, it follows that $\restr{f_k^i}{I(u)}$ is an isometry. Hence for every $x\in I(0^k \lp n)$ with $0\leq n < 2^k$, $1\leq\ell\leq\infty$ and $0\leq i, j < \infty$
\begin{equation}
\label{eq:rovnost-miery}
\dis_{\ell}(f_k^i(x), f_k^j(x)) = \dis_{\ell}(f_k^{i + n}(0), f_k^{j + n}(0))
\end{equation}
Obviously, $\dis_\ell(f_k^i(x), f_k^j(x)\leq \ve$ if and only if $\dis_\ell(f_k^{i + n}(0), f_k^{j + n}(0)) \leq \ve$ for every $\ve > 0$. 

Let $\mu_{k, x}$ be the uniform ergodic measure on $\Orb_{f_k}(x)$, consequently $\mu_{k, x}$ is positive on the $f_k$--orbit of $x$. The measure $\mu_k$ is uniform and positive on $\Orb_{f_k}(0)$. By~(\ref{eq:rovnost-miery}),
\begin{align}
\label{eq:per_pts}
\mu_k\times\mu_k\{(x, y)\in[0, 1]^2\ |\ \dis_\ell(x, y)\leq \ve\} = \mu_{k, z}\times\mu_{k, z}\{(x, y)\in[0, 1]^2\ |\ \dis_{\ell}(x, y)\leq\ve\}.
\end{align}

\begin{thm}
\label{thm:ostatok_body}
For all $x\in [0, 1]$ and $1\leq\ell\leq\infty$, the limits of $(\tC{\ell}{f_{\alpha}}(x, i, \ve))_i$ and $(\tDET{\ell}{f_{\alpha}}(x, i, \ve))_i$ exist. Moreover, 
\begin{itemize}
	\item $\tc{\ell}{f}(x, \ve) = \tc{\ell}{f}(\ve)$ and $\tdet{\ell}{f}(x, \ve) = \tdet{\ell}{f}(\ve)$ for all $x\notin \eper(f)$ and 
	\item $\tc{\ell}{f}(x, \ve) = \tc{\ell}{f_{k}}(\ve)$ and $\tdet{\ell}{f}(x, \ve) = \tdet{\ell}{f_k}(\ve)$ for $x\in\eper(f)$ of period $2^k$ where $k\geq 0$.
\end{itemize}
\end{thm}
\begin{proof}
Let $x\in\eper(f)$. Since the beginnings of trajectories are not important for the limits, we can assume that $x$ is periodic. From Lemma~\ref{lema:per_del}, $f_k^i(x) = f^i(x)$ for $i\in\ZZZ$ and from Corollary~\ref{cor:per_vyskyt}, $x\in X_k$. Then the latter statement follows from~(\ref{eq:per_pts}).

Now let $x\notin\eper(f)$. All such points are eventually periodic for every $f_k$. Since all of measures $\mu_{k, z}$ on these points are zero, we can ignore the beginning of the orbit. From~(\ref{eq:per_pts}),
\begin{align}
\label{eq:every_x}
\lim_{n\to\infty}\tC{\ell}{f_k}(x, n, \ve) = \lim_{n\to\infty}\tC{\ell}{f_k}(0, n, \ve) = \tc{\ell}{f_k}(\ve).
\end{align}

The map $f_k$ is the same as $f$ on $[0, 1 - \alpha^{k - 1} + \alpha^k]$. The functions are different on $(1 - \alpha^{k - 1} + \alpha^k, 1]$. Consider $I(1^{k - 1}) \supset (1 - \alpha^{k - 1} + \alpha^k, 1]$ which is periodic for all of $f, f_{k - 1}$ and $f_k$.

If $x = 1$, then $f(x) = 0$ and $\lim_{n\to\infty}\tC{\ell}{f}(1, n, \ve) = \lim_{n\to\infty}\tC{\ell}{f}(0, n, \ve) = \tc{\ell}{f}(\ve)$. Therefore we can assume that $x\neq 1$. Let $k > 0$ be large enough that there exists $h\geq 0$ such that the iterations $f^i(x)\notin I(1^{k - 1})$ for $0\leq i\leq h$. Suppose that $h$ is maximal such integer, i.e.~$f^{h + 1}(x)\in I(1^{k - 1})$. Then $f^i(x) = f_k^i(x)$ for all $0\leq i \leq h$, but $f^j(x)$ need not be the same as $f_k^j(x)$ for $j > h$. For every $j \geq 0$, the points $f^{h + j + 1}(x), f_k^{h + j + 1}(x)\in f^j(I(1^{k - 1})) = f_k^j(I(1^{k - 1}))$. From that, there is $u\in\Sigma^{k - 1}$ such that $f^i(x), f_k^i(x)\in f^i(I(u))$ and $\dis(f^i(x), f_k^i(x))\leq\alpha^{k - 1}$ for every $i > h$.

Let $\dis(f^i(x), f^j(x)) \leq \ve$. Therefore
\begin{align*}
\dis(f_k^i(x), f_k^j(x))\leq \dis(f_k^i(x), f^i(x)) + \dis(f^i(x), f^j(x)) + \dis(f^j(x), f_k^j(x))\leq \ve + 2\cdot\alpha^{k - 1}.
\end{align*}

Similarly, if $\dis(f_k^i(x), f_k^j(x))\leq \ve - 2\cdot \alpha^{k - 1}$, then $\dis(f^i(x), f^j(x))\leq \ve$. From~(\ref{eq:dvaja_policajti}),
\begin{align*}
\tC{\ell}{f_k}(x, n, \ve - 2\cdot\alpha^{k - 1}) \leq \tC{\ell}{f}(x, n, \ve)\leq \tC{\ell}{f_k}(x, n, \ve + 2\cdot \alpha^{k - 1})
\end{align*}
and from~(\ref{eq:every_x}),
\begin{align*}
&\lim_{n\to\infty}\tC{\ell}{f_k}(x, n, \ve - 2\cdot\alpha^{k - 1}) = \tc{\ell}{f_k}(\ve - 2\cdot\alpha^{k - 1})\qquad\text{and}\qquad\\
&\lim_{n\to\infty}\tC{\ell}{f_k}(x, n, \ve + 2\cdot\alpha^{k - 1}) = \tc{\ell}{f_k}(\ve + 2\cdot\alpha^{k - 1}).
\end{align*}
Thus $\tc{\ell}{f}(x, \ve)\subset [\tc{\ell}{f_k}(\ve - 2\cdot\alpha^{k - 1}), \tc{\ell}{f_k}(\ve + 2\cdot\alpha^{k - 1})]$ for every large enough $k$.

We finish the proof by using Lemma~\ref{lema:f_k_prechod_f} where $g(k) = 2\cdot\alpha^{k - 1}$. In an interval $J\subset [0, 1]$ of length $\alpha^{k - 1}$, there are at most three words $u, u\pp 1, u \pp 2\in\Sigma^k$ satisfying $J\cap I(v)\neq\emptyset$ where $v\in\{u, u\pp 1, u\pp 2\}$. Then $h_{g(k)} = 6$ for each $k$ and clearly $6/2^k\to 0$ as $k\to\infty$. Using Lemma~\ref{lema:f_k_prechod_f} and Lemma~\ref{lema:c_funkcia}, we then actually have
\begin{align*}
\bigcap_{k > 0}\left[\tc{\ell}{f_k}(\ve - 2\cdot\alpha^{k - 1}), \tc{\ell}{f_k}(\ve + 2\cdot\alpha^{k - 1})\right] = \left\{\tc{\ell}{f}(\ve)\right\}
\end{align*}
which finishes the proof.
\end{proof}

The correlation integrals and asymptotic determinisms change continuously not only with respect to $\ve$ but for a fixed $\ve$ these functions are continuous at every $\alpha\in(0, 1/2)$ as well.

\begin{lema}
\label{lema:int_conti}
For every $1\leq\ell\leq\infty$, the correlation integral $\tc{\ell}{f_{\alpha}}(\ve)$ and asymptotic determinism $\tdet{\ell}{f_{\alpha}}(\ve)$ change continuously with respect to $\alpha$.
\end{lema}
\begin{proof}
Let $1\leq\ell\leq\infty$ and $\varepsilon > 0$. From the proof of Lemma~\ref{lema:f_k_prechod_f}, there is $k$ such that $|\tc{\ell}{f_{\alpha, k}}(\ve) - \tc{\ell}{f_{\alpha, k}}(\ve')| < \varepsilon/3$ for every $\alpha\in (0, 1/2)$ and $\ve'\in (\ve - \alpha^k, \ve + \alpha^k)$. Then from Corollary~\ref{cor:rozdiel_ck_c}, $|\tc{\ell}{f_{\alpha}}(\ve) - \tc{\ell}{f_{\alpha, k}}(\ve)| < \varepsilon/3$ for all $\ve > 0$ and $0 < \alpha < 1/2$.

Fix $\alpha$ and $\ve$. If there are $u, v\in\Sigma^k$ such that $\dis^{\alpha}(u, v) = \ve$, choose $\ve < \ve' < \ve + (\alpha/2)^k$ to satisfy $D_{k, \alpha}(\ve) = D_{k, \alpha}(\ve')$ and $\dis^{\alpha}(u, v)\neq \ve'$ for every pair $u, v\in\Sigma^k$. Hence $\tc{\ell}{f_{\alpha, k}}(\ve) = \tc{\ell}{f_{\alpha, k}}(\ve')$.

By~(\ref{eq:gap_seq_dis}), for every $u \neq v\in\Sigma^k$ there are $a, b\in\NNN\cup \{0\}$ such that $\dis^{\alpha}(u, v) = \alpha^k\cdot a + \alpha^{k - 1}\cdot(1 - 2\alpha)\cdot b$ for every $\alpha$. The function $g_{u, v}:\RRR\to\RRR$, where $g_{u, v}(x) = x^k\cdot a + x^{k - 1}\cdot(1 - 2x)\cdot b$, is continuous, thus there is some $\delta_{u, v} > 0$ such that for every $\alpha'\in(\alpha - \delta_{u, v}, \alpha + \delta_{u, v})$ the product $(g_{u, v}(\alpha) - \ve')\cdot(g_{u, v}(\alpha') - \ve') > 0$.
Set $\delta = \min\{\min\{\delta_{u, v}\ |\ u\neq v\in\Sigma^k\}, \alpha/2\} > 0$. Then, $\tc{\ell}{f_{\alpha, k}}(\ve) = \tc{\ell}{f_{\alpha, k}}(\ve') = \tc{\ell}{f_{\alpha', k}}(\ve')$ for every $\alpha'\in (\alpha - \delta, \alpha + \delta)$. The radius $\ve' < \ve + (\alpha/2)^k < \ve + (\alpha')^k$ and therefore $\tc{\ell}{f_{\alpha', k}}(\ve') - \tc{\ell}{f_{\alpha', k}}(\ve) = \tc{\ell}{f_{\alpha, k}}(\ve) - \tc{\ell}{f_{\alpha', k}}(\ve) < \varepsilon/3$.

Hence for every $\varepsilon > 0$, $\alpha\in (0, 1/2)$ and $\ve\in (0, 1)$, there are $k > 0$, $\delta > 0$ such that for every $\alpha'\in (\alpha - \delta, \alpha + \delta)$,
\begin{align*}
\left|\tc{\ell}{f_{\alpha'}}(\ve) - \tc{\ell}{f_{\alpha}}(\ve)\right| &\leq \left|\tc{\ell}{f_{\alpha'}}(\ve) - \tc{\ell}{f_{\alpha', k}}(\ve)\right| + \left|\tc{\ell}{f_{\alpha', k}}(\ve) - \tc{\ell}{f_{\alpha, k}}(\ve)\right| + \left|\tc{\ell}{f_{\alpha, k}}(\ve) - \tc{\ell}{f_{\alpha}}(\ve)\right| < \\
& < \frac{\varepsilon}{3} + \frac{\varepsilon}{3} + \frac{\varepsilon}{3} = \varepsilon.
\end{align*}
Thus the function $\tc{\ell}{f_{\alpha}}(\ve)$ is for fixed $\ell, \ve$ continuous in every $\alpha \in (0, 1/2)$. Therefore, from the definition of asymptotic determinism, $\tdet{\ell}{f_{\alpha}}(\ve)$ is continuous as well.
\end{proof}

Even though all $\tc{\ell}{f_{\cdot}}(\ve)$ are continuous, $\tc{\ell}{f_{\cdot, k}}(\ve)$ are not continuous in general. Let $k > 0$, $\alpha\in(0, 1/2)$ and $\ve\in(0, 1)$ be such that for some $u, v\in\Sigma^k$ the distance $\dis^\alpha(u, v) = \ve$. Now, let $\varepsilon < 1/2^{2k}$. If $\sum_{1\leq i, j \leq 2^k}D_{k, \alpha}[i, j](\ve)\neq \sum_{1\leq i, j\leq 2^k}D_{k, \alpha'}(\ve)[i, j]$, then $|\tc{1}{f_{\alpha, k}}(\ve) - \tc{1}{f_{\alpha', k}}(\ve)| > \varepsilon$. The above-mentioned function $g_{u, v}$ is continuous, nowhere constant and increasing for $x \geq 0$. Hence if $\dis^{\alpha}(u, v) = \ve$, then for every sufficiently small neighbourhood of $\alpha$ there is some $\alpha'$ satisfying $\sum_{1\leq i, j \leq 2^k}D_{k, \alpha}[i, j](\ve)\neq \sum_{1\leq i, j\leq 2^k}D_{k, \alpha'}(\ve)[i, j]$. Let $\ell\leq\infty$, and $\ve > 0$, next, let $u, v\in\Sigma^k$ be such that $\dis_{\ell}^{\alpha}(u, v) = \ve$. The map $g_{u, v, \ell} = \max\{g_{u\lp i, v\lp i}\ |\ 0\leq i < \ell\}$ is continuous. Since $g_{u, v}$ is nowhere constant for all $u, v\in\Sigma^k$, each $g_{u, v, \ell}$ is nowhere constant and increasing for $x\geq 0$. Then the proof for discontinuity of $\tc{\ell}{f_{\cdot, k}}(\ve)$ follows the case $\ell = 1$ where we use $D_{k, \alpha, \ell}$ instead of $D_{k, \alpha}$.
\newline

The next lemma is very helpful for computations of $\lim\inf_{\ve\to 0^+}\tdet{\ell}{f}(\ve)$ and $\lim\sup_{\ve\to 0^+}\tdet{\ell}{f}(\ve)$ which is the main purpose of this paper.

\begin{lema}
\label{lema:det_epsilon}
For each $\ve \in (0, 1)$,
\begin{align*}
\tdet{\infty}{f_{\alpha}}(\alpha\cdot\ve)\leq \tdet{\infty}{f_{\alpha}}(\ve).
\end{align*}
Moreover, for $\ve \leq (1 - 2\alpha)/\alpha$, the equality holds.
\end{lema}
\begin{proof}
Let $u\in\Sigma^k$ be given. By an easy computation, $\gamma(0u) = \gamma(u)$ and $\gamma(1 u) = \gamma(u) + 2^k$. From~(\ref{eq:gap_seq_dis}), for every $u, v\in\Sigma^k$,
\begin{align*}
\dis(0u, 0v) &= (\gamma(v) - \gamma(u))\cdot\alpha^{k + 1} + \alpha^k\cdot(1 - 2\alpha)\cdot A(\gamma(u), \gamma(v) - \gamma(u)) = \alpha\cdot\rho(u, v)\qquad\text{and}\\
\dis(1u, 1v) &= (\gamma(v) - \gamma(u))\cdot\alpha^{k + 1} + \alpha^k\cdot(1 - 2\alpha)\cdot A(\gamma(u) + 2^k, \gamma(v) - \gamma(u)).
\end{align*}
By~(\ref{lema:a_s}), we get
\begin{align*}
A(\gamma(u), \gamma(v) - \gamma(u)) = A(\gamma(u) + 2^k, \gamma(v) - \gamma(u)).
\end{align*}
Therefore
\begin{equation}
\label{eq:new_dist}
\dis(u, v) = \alpha\dis(0u, 0v) = \alpha\dis(1u, 1v)
\end{equation}
and $\tc{1}{f_{\alpha, k + 1}}(\alpha\ve) \geq 1/2\cdot \tc{1}{f_{\alpha, k}}(\ve)$.

Let $D_{k + 1}(\alpha\cdot\ve)[2^k, 2^k + 1] = 0$, i.e.~$\dis(0 1^k, 1 0^k) = \alpha^{k + 1} + (1 + 2\alpha) > \alpha\cdot\ve$. Then, from the pattern $A_0$ and from~(\ref{eq:new_dist}), the distance $\dis(w u, w' v)$ with $w, w'\in\{0, 1\}$, is for $\alpha\cdot\ve \leq 1 - 2\alpha$ less than or equal to $\alpha\cdot\ve$ if and only if $w = w'$ and $\dis(u, v)\leq \ve$.

By Lemma~\ref{dosledok:pocet_blizke_orbity}, for computing $\tc{\infty}{f_{\alpha, k}}(\ve)$ it is sufficient to study only the first row in the distance matrix $D_k(\ve)$. If $0^h 1 1 u$ where $u\in\Sigma^{k - h - 2}$ is the word from the lemma for $k$ and $\ve$, then by~(\ref{eq:new_dist}), $0^{h + 1} 1 1 u$ is the word for $k + 1$ and $\alpha\cdot\ve$. Thus $\tc{\infty}{f_{\alpha, k + 1}}(\alpha\ve) = 1/2\cdot \tc{\infty}{f_{\alpha, k}}(\ve)$ for every $\ve > 0$ and $k\geq 1$.

Therefore
\begin{align*}
\tdet{\infty}{f_{\alpha, k + 1}}(\alpha\ve) = \frac{\tc{\infty}{f_{\alpha, k + 1}}(\alpha\ve)}{\tc{1}{f_{\alpha, k + 1}}(\alpha\ve)} \leq \frac{\tc{\infty}{f_{\alpha, k}}(\ve)}{\tc{1}{f_{\alpha, k}}(\ve)} = \tdet{\infty}{f_{\alpha, k}}(\ve),
\end{align*}
and the equality holds if $\alpha\ve \leq 1 - 2\alpha$. Applying Theorem~\ref{thm:prechod} we complete the proof.
\end{proof}

From Lemma~\ref{lema:det_epsilon}, if $\alpha \leq 1/3$, it is sufficient to compute correlation integrals only for $\ve\in (\alpha, 1]$.

\begin{figure}[!ht]
  \centering
  \includegraphics[width=1\textwidth]{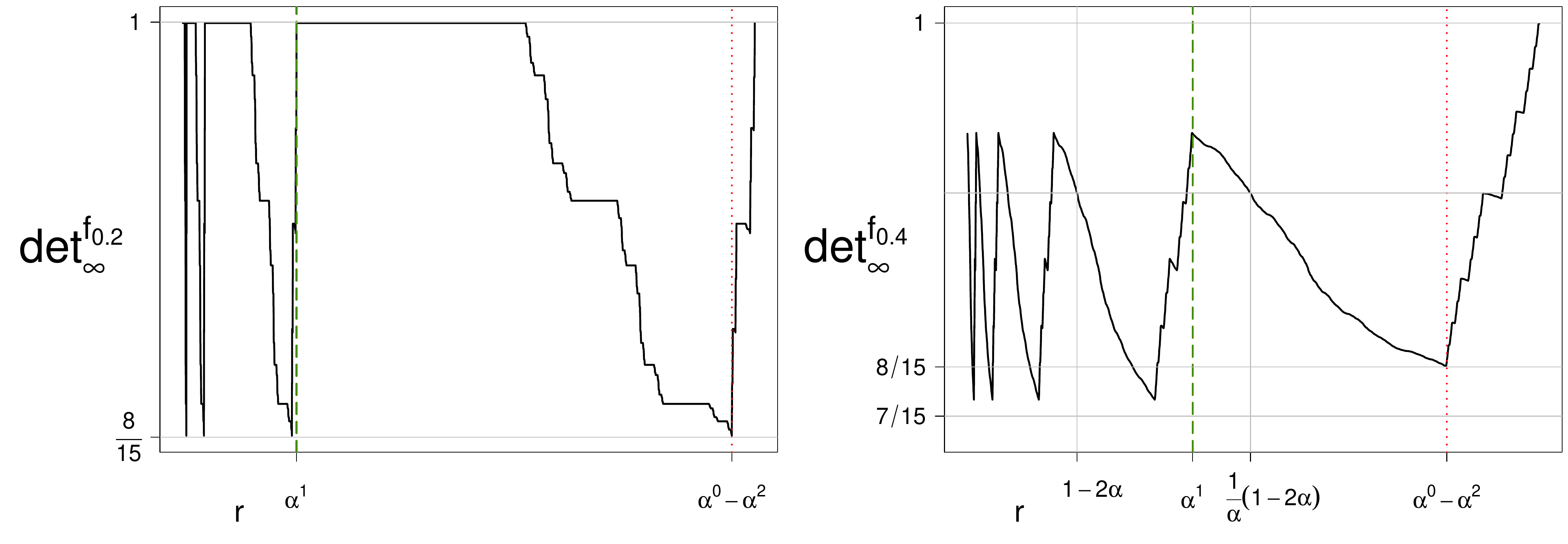}
  \caption[Determinism as a function of $\ve$]{Examples of determinism as a function of $\ve$ for $\alpha = 0.2$ and $\alpha = 0.4$.}
  \label{fig:determinism}
\end{figure}

Define the functions $\utdet, \otdet:(0, 1/2)\to (1/3, 1]$ by $\utdet(\alpha) = \lim\inf_{\ve\to 0^{+}}\tdet{\infty}{f_{\alpha}}(\ve)$ and $\otdet(\alpha) = \lim\sup_{\ve\to 0^{+}}\tdet{\infty}{f_{\alpha}}(\ve)$. The next theorem gives some bounds for these.

\begin{thm}
\label{thm:determinizmus_ohranicenie_tretina}
The local minima of infinite determinism are located at points $\alpha^{h} - \alpha^{h + 2}$ where $h\geq 0$. Then $\utdet(\alpha)\in [1/3, 8/15]$ for all $\alpha$. The maxima $\otdet(\alpha) = 1$ for $\alpha \leq 1/3$ and $\otdet(\alpha) < 1$ for $\alpha > 1/3$. If $\alpha\leq 1/3$, the extremes are $8/15$ and $1$ and maxima are achieved at $\ve = \alpha^h$.
\end{thm}
\begin{proof}
Assume that $\alpha\leq 1/3$ and $\ve\in (\alpha, 1]$, i.e.~$h = 0$. Then by Lemma~\ref{lema:det_epsilon}, the resulting minima and maxima on $(\alpha, 1]$ will be the maxima and minima on $(0, 1]$ as well. Moreover, if the minima for $\tdet{\infty}{f_{\alpha}}$ is achieved at $\ve_{\min}\in(\alpha, 1]$, then the minima are achieved at each $\alpha^m\cdot\ve_{\min}$, $m \geq 0$. The same holds for the maxima.

Let $k > 0$ and $u\in\Sigma^k$ be such that $\dis(0^k, u)\leq \ve$ and $\dis(0^k, u \pp 1) > \ve$. Let $v\in\Sigma^k$ satisfy $\dis(0^k, v)\leq \ve$ and either
\begin{itemize}
	\item $\dis(0^k, v \pp 1) > \ve$ or 
	\item $v = 0 1^{k + 1}$ and $\dis(0^k, 1 1 0^{k - 2}) > \ve$. 
\end{itemize}
Set $j = \gamma(u)$ and $j' = \gamma(v)$. From the definition of $v$, if $j\in\{2^{k - 1}, 2^{k - 1} + 2^{k - 2} + 1, 2^{k - 1} + 2^{k - 2} + 2, \ldots, 2^k\}$, then $j' = j$. Otherwise we have $j' = 2^{k - 1}$.

From Corollary~\ref{dosledok:pocet_blizke_orbity}, if $j' = 2^{k - 1}$, then the correlation integral $\tc{\infty}{f_k}(\ve) = 1/2$. If $j' = 2^{k - 1} + 2^{k - 2} + n$, $1\leq n \leq 2^{k - 2}$, then $\tc{\infty}{f_k}(\ve) = 1/2 + n/2^{k - 1}$.

Following notation from Lemma~\ref{lema:hranice_rr}, we have $j_1 = 2^k - j$ and
\begin{align*}
\frac{1}{2^{2k}}\left[2(2^{2k - 1} - j_1^2)\right]\leq \tc{1}{f_k}(\ve)\leq \frac{1}{2^{2k}}\left[2^{2k} - j_1^2\right].
\end{align*}

For the considered $\alpha$, $\dis(0^k, 0 1^{k - 1}) \leq \alpha < \dis(0 1^{k - 1}, 1 0^{k - 1})$ and $\dis(0^k, 1 0 1^{k - 2})\leq 1 - \alpha + \alpha^2 < \dis(0^2 1^{k - 2}, 1^2 0^{k - 2})$. From the patterns $A_0$ and $A_1$, we thus get $\tc{1}{f_k}(\alpha) = \tc{\infty}{f_k}(\alpha) = \tc{\infty}{f_k}(1 - \alpha + \alpha^2) = 1/2$ and $\tc{1}{f_k}(1 - \alpha + \alpha^2) = 1 - 2\cdot (2^{k - 2})^2/2^{2k}$.

Similarly, $\dis(0 1^{k - 1}, 1^k)\leq 1 - \alpha < \dis(0^k, 1 0^{k - 2} 1)$ and $\dis(0^2 1^{k - 2}, 1^k) \leq 1 - \alpha^2 < \dis(0^k, 1^2 0^{k - 3} 1)$. From the patterns $B_0$ and $C_1$, we thus have $\tc{1}{f_k}(1 - \alpha) = 1 - (2^{k - 1} - 1)^2/2^{2k}$, $\tc{\infty}{f_k}(1 - \alpha) = 1/2$, $\tc{1}{f_k}(1 - \alpha^2) = 1 - (2^{k - 2} - 1)^2/2^{2k}$ and $\tc{\infty}{f_k}(1 - \alpha^2) = 1/2 + 1/2^{k - 1}$.

For better visualization of distance matrices for these values see Figure~\ref{fig:border_cases}.

\begin{figure}[!ht]
	\centering
				\includegraphics[width=0.9\textwidth]{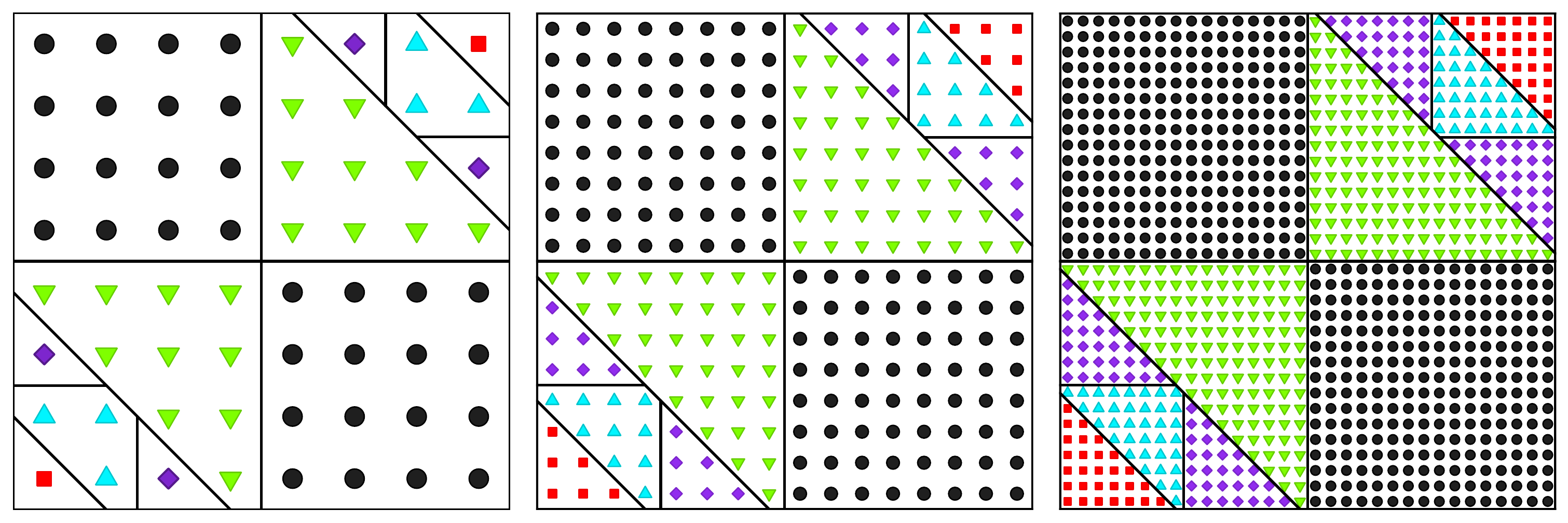}
  \caption[Examples of distance matrices for some special values]{Visualization of the distance matrix for $\alpha\leq 1/3$ and $\ve\in (\alpha, 1]$. There are black dots for $\ve = \alpha$, add green triangles pointed down for $\ve = 1 - \alpha$, add purple diamonds for $\ve = 1 - \alpha + \alpha^{2}$, add blue triangles pointed up for $\ve = 1 - \alpha^{2}$ and similarly as in the case of black dots, add red squares for $\ve = 1$. Matrices shown here are constructed for the cases $k = 3$, $k = 4$ and $k = 5$.}
    \label{fig:border_cases}
\end{figure}

From Lemma~\ref{lema:f_k_prechod_f}, we see that with $k\to\infty$,
\begin{eqnarray*}
\begin{array}{lll}
\tc{1}{f}(\alpha) = \frac{1}{2},                   &   \tc{\infty}{f}(\alpha) = \frac{1}{2},     			       				& \tdet{\infty}{f}(\alpha) = 1,\\
\tc{1}{f}(1 - \alpha) = \frac{3}{4},               &   \tc{\infty}{f}(1 - \alpha) = \frac{1}{2},    	           		& \tdet{\infty}{f}(1 - \alpha) = \frac{2}{3}, \\
\tc{1}{f}(1 - \alpha + \alpha^{2}) = \frac{7}{8},  &   \tc{\infty}{f}(1 - \alpha + \alpha^{2}) = \frac{1}{2},      	& \tdet{\infty}{f}(1 - \alpha + \alpha^{2}) = \frac{4}{7}, \\
\tc{1}{f}(1 - \alpha^{2}) = \frac{15}{16},         &   \tc{\infty}{f}(1 - \alpha^{2}) = \frac{1}{2},               	& \tdet{\infty}{f}(1 - \alpha^{2}) = \frac{8}{15}.
\end{array}
\end{eqnarray*}

The determinism cannot be larger than $1$, therefore we see that for $0 < \alpha\leq 1/3$ the maximum is at $\ve = \alpha^h$.

Since for $\ve\in[\alpha, 1 - \alpha^{2}]$ the integral $\tc{\infty}{f}(\ve)$ remains the same and $\tc{1}{f}(\ve)$ increases, determinism is decreasing for these $\ve$. It remains to show that $\tdet{\infty}{f}(\ve) > 8/15$ for every $\ve\in (1 - \alpha^{2}, 1]$. In these values $0\leq j_1 < 2^{k - 2}$, $\tc{\infty}{f_{k}}(\ve) = 1 - 2j_1/2^{k}$ and $\tc{1}{f_{k}}(\ve) \leq 1/2^{2k}\left[2^{2k} - j_1^2\right]$. Therefore
\begin{align*}
\tdet{\infty}{f_{k}}(\ve) \geq \frac{2^{k + 1}(2^{k - 1} - j_1)}{(2^k - j_1)(2^k + j_1)}.
\end{align*}
The $j_1$--differentiation and Lemma~\ref{lema:f_k_prechod_f} prove the case $\alpha\leq 1/3$.

Now, let $1/3 < \alpha < 1/2$, $0 < \ve \leq 1$, and $k > 0$ be large enough. First, we prove that $\tdet{\infty}{f}(\ve) \geq 1/3$. Let $s\geq 0$ be such that $\tdet{\infty}{f_\alpha}(\alpha^{s + 1}) = \tdet{\infty}{f_\alpha}(\alpha^{s})$, i.e.~$\alpha^s\leq (1 - 2\alpha)/\alpha$. Similarly as was noted before, by Lemma~\ref{lema:det_epsilon}, the resulting maxima and minima of $\tdet{\infty}{f_{\alpha}}$ on $(\alpha^{s + 1}, \alpha^s]$ is $\otdet(\alpha)$ and $\utdet(\alpha)$.

Let $u\in\Sigma^k$ be such that $I(0^k)$ and $I(u)$ are $\ve$--close while $I(0^k)$ and $I(u \pp 1)$ are $\ve$--distant. From the pattern $A_0$, the correlation integrals $\tc{1}{f_k}(\ve)$ are always less than or equal to $1 - (1 - j/2^k)^2 = (2^{k + 1}j - j^2)/2^{2k}$ where $j = \gamma(u)$. We show that $\tdet{\infty}{f_k}(\ve) > 1/3$.

Suppose that $2^{k - s - 1} < j\leq 2^{k - s - 1} + 2^{k - s - 2} < 2^k$. In this case $\tc{\infty}{f_{k}}(\ve) = 2^{2k - s - 1}/2^{2k} = 2^{- s - 1}$ and $\tdet{\infty}{f_k}(\ve) \geq 2^{2k - s - 1}/(2^{k + 1}j - j^2)$. Since $\tc{1}{f_{k}}(\ve)$ is increasing for these $j$, the smallest asymptotic determinism is for $j = 2^{k - s - 1} + 2^{k - s - 2}$. In this case $\tdet{\infty}{f_{k}}(\ve) \geq 2^{s + 3}/(5\cdot(2^{s + 2} - 5)) > 1/3$. For $2^{k - s - 1} + 2^{k - s - 2} < j \leq 2^{k - s}$, the integral $\tc{\infty}{f_k}(\ve) = (2j - 2^{k - s})/2^{k}$ and determinism $\tdet{\infty}{f_k}(\ve) \geq g(j) = 2^k(2j - 2^{k - s})/(2^{k + 1}j - j^2)$. A simple $j$--differentiation shows that $g$ is an increasing function of $j$. For $j = 2^{k - s - 1} + 2^{k - s - 2} + 1$, the value $g(j) > 1/3$. Therefore $\tdet{\infty}{f}(\ve) \geq 1/3$ for every $\alpha$ and every $\ve$.

Note that $j = 2^{k - s - 1} + 2^{k - s - 2}$ is obtained at $\ve\in [\alpha^{s} - \alpha^{s + 1} + \alpha^{s + 2} - \alpha^k, \alpha^{s} - \alpha^{s + 2})$. Obviously, the larger $\ve$, the larger $\tc{1}{f_k}(\ve)$. Since $\tc{\infty}{f_k}$ remains the same, the asymptotic determinism in the infinite horizon decreases for $\ve\to\alpha^{s} - \alpha^{s + 2}$. Similarly, $j = 2^{k - s - 1} + 2^{k - s - 2} + 1$ is obtained at $\ve\in[\alpha^{s} - \alpha^{s + 2}, \alpha^{s} - \alpha^{s + 2} + \alpha^{k - 1} - \alpha^k)$ and the determinism decreases for $\ve\to \alpha^{s} - \alpha^{s + 2} + \alpha^{k - 1} - \alpha^k$. The interval $(\alpha^{s} - \alpha^{s + 2} - \alpha^k, \alpha^{s} - \alpha^{s + 2} + \alpha^{k - 1} - \alpha^k)\subset (\alpha^{s} - \alpha^{s + 2} - \alpha^{k - 1}, \alpha^{s} - \alpha^{s + 2} + \alpha^{k + 1})$ for every $k$. Using Lemma~\ref{lema:f_k_prechod_f} we finish the proof that the determinism achieves its minimum at $\ve = \alpha^s - \alpha^{s + 2}$ and its value is at least $1/3$.

Now, let $h \geq 0$ and $k > 0$ be as in the proof of Lemma~\ref{lema:existence_pair}, moreover let $u, v\in\Sigma^k$ be the desired words from that lemma. Then $\tdet{\infty}{f_k} \leq 1 - 1/(2^{2k}\cdot \tc{1}{f_k}(\ve)) < 1 - 1/2^{2k}$. Let $m > 0$. By Lemma~\ref{lema:pocet_daleke_bodky}, there is at least $2^{m - 1}\cdot(1 + 2^m)$ words $u, v\in\Sigma^{k + m}$ satisfying Lemma~\ref{lema:existence_pair}. Then $\tdet{\infty}{f_{k + m}}(\ve)\leq 1 - (2^{m - 1}(1 + 2^m))/(2^{2(k + m)}\tc{\infty}{f_{k + m}}) < 1 - 1/2^{2k + 1}$. By Theorem~\ref{thm:prechod}, the determinism $\tdet{\infty}{f}(\ve) \leq 1 - 1/2^{2k + 1}$.
\end{proof}

In the proof of the next theorem it is more convenient to use the definition of determinism via recurrence rates.

\begin{dosledok}[of Theorem~\ref{thm:ostatok_body}]
\label{cor:rr_converge}
For every $x\in[0, 1]$ and $1\leq\ell\leq\infty$, the limit of $(\tRR{\ell}{f}(x, i, \ve))_i$ exists and we denote it by $\trr{\ell}{f}(x, \ve)$. Moreover, it is equal to 
\begin{itemize}
	\item $\ell\cdot \tc{\ell}{f}(x, \ve) - (\ell - 1)\cdot \tc{\ell + 1}{f}(x, \ve)$ for $\ell < \infty$ and 
	\item $\tc{\ell}{f}(x, \ve)$ for $\ell = \infty$. 
\end{itemize}
The asymptotic determinism $\tdet{\ell}{f}(x, \ve)$ equals $\trr{\ell}{f}(x, \ve)/\trr{1}{f}(x, \ve)$.
\end{dosledok}

By Corollary~\ref{thm:ostatok_body}, if $x_1\neq x_2\in X_0$, then $\tdet{\ell}{f}(x_1, \ve)$ need not be the same as $\tdet{\ell}{f}(x_2, \ve)$. By Lemma~\ref{lema:per_det_one}, Theorem~\ref{thm:determinizmus_ohranicenie_tretina} and Theorem~\ref{lema:per_det_one}, for $\ell = \infty$ there are in fact such pairs of points for which the equality does not hold. The next theorem shows that for $\ell < \infty$ such pairs may not exist.

\begin{thm}
\label{thm:det_jeden}
For every $0 < \ell < \infty$, there is $\ve_0$ satisfying $\tdet{\ell'}{f}(x, \ve) = 1$ for every $0 < \ve < \ve_0$, $0 < \ell'\leq\ell$ and $x\in X_0$.
\end{thm}
\begin{proof}
From Corollary~\ref{cor:rr_converge}, it is sufficient to prove that for every $1\leq \ell < \infty$ there is $\ve_0 > 0$ with the property that $\tRR{\ell}{f}(x, n, \ve) = \tRR{1}{f}(x, n, \ve)$ for every $x\in [0, 1]$ and $0 < \ve \leq \ve_0$. Fix $\ell < \infty$. From the definition of $\ell$--recurrence rate, it is clear that $\tRR{\ell}{f}(x, n, r) \leq \tRR{1}{f}(x, n, \ve)$. Therefore we prove only the opposite inequality.

Let $h$ be such that $\ell\leq 2^h$ and $\ve \leq \ve_0 < (1 - 2\alpha)\cdot\alpha^{h - 1}$ be given.

First, consider $\tRR{\ell}{f}(0, n, \ve)$ and $\tRR{1}{f}(0, n, \ve)$. If we prove their equality, recurrence rates are equal for every $x\notin\eper(f)$.

Let $0\leq i, j < n$. Suppose that $d_E(f^i(0), f^j(0))\leq\ve < (1 - 2\alpha)\cdot\alpha^{h - 1}$. Thus, there are $u\in\Sigma^h$ and $v, w\in\Sigma^\infty$ such that $f^i(0) = \cores(u v)$ and $f^j(0) = \cores(u w)$. From Lemma~\ref{lema:trajektorie_dvoch_bodov_rozsah},
\begin{align*}
\dis_{k}(0^h v, 0^h w) = \dis_{2^h}(0^h v, 0^h w) = \dis(u v, u w) = d_E(f^i(0), f^j(0))
\end{align*}
for every $k\leq 2^h$. Hence there is $m\leq\min\{i, j, \ell - 1\}$ satisfying $\dis_{\ell}(u v \lm m, u w \lm m) = d_E(f^i(0), f^j(0))\leq\ve$. From the definition of $\ell$--recurrence rate~(\ref{eq:rec_sum}), we thus get $\tRR{1}{f}(0, n, \ve) = \tRR{\ell}{f}(0, n, \ve)$.

Now, let $x\in\eper(f)$ be of period $2^k$. By Corollary~\ref{cor:per_vyskyt}, $x\in\eper(f)\cap (X_k\setminus X_{k + 1})$. By Lemma~\ref{lema:periodic} and Lemma~\ref{lema:korelacny_rekurent}, we can assume that $n = 2^k$.  By Lemma~\ref{lema:per_del},
\begin{align*}
\min\{d_E(f^i(x), f^j(x))\ |\ 0\leq i\neq j < 2^k\} \geq \alpha^{k - 1}(1 - 2\alpha) + 2\alpha^{k + 1} > \alpha^{k - 1}(1 - 2\alpha).
\end{align*}
If $k \leq h$, then $\dis(f^i(x), f^j(x))\leq \ve$ implies $f^i(x) = f^j(x)$. Thus, $\tRR{\ell}{f}(x, n, \ve) = \tRR{1}{f}(x, n, \ve)$. Suppose that $k > h$. We can repeat the proof for $x = 0$, but we use the words $v, w$ of length $k - h$ instead of infinitely long words. Therefore, there are $u\in\Sigma^h$ and $v, w \in\Sigma^{k - h}$ such that $f^i(x)\in I(u v)$ and $f^j(x)\in I(u w)$ for  $0\leq i, j < n$.
\end{proof}

By results from this section, all of the functions
\begin{align*}
\mathrm{c}: (0, 1/2)\times[0, 1]\times(0, 1]\times(\NNN\cup\{\infty\})&\to [0, 1]\\
		(\alpha, x, \ve, \ell)&\mapsto \lim_{n\to\infty}\tC{\ell}{f_{\alpha}}(x, n, \ve)\\
\mathrm{det}: (0, 1/2)\times[0, 1]\times(0, 1]\times(\NNN\cup\{\infty\})&\to [0, 1]\\
		(\alpha, x, \ve, \ell)&\mapsto \lim_{n\to\infty}\tDET{\ell}{f_{\alpha}}(x, n, \ve)
\end{align*}
are well defined. None of $\mathrm{c}(\cdot, x, \ve, \ell)$, $\mathrm{c}(\alpha, \cdot, \ve, \ell)$, $\mathrm{c}(\alpha, x, \cdot, \ell)$, and $\mathrm{det}(\cdot, x, \ve, \ell)$, $\mathrm{det}(\alpha, \cdot, \ve, \ell)$, $\mathrm{det}(\alpha, x, \cdot, \ell)$ is continuous in general. If $\ve = 1$, all correlation integrals and asymptotic determinisms are equal to $1$ and therefore these are continuous. Suppose that $\ve\neq 1$. If the argument is 
\begin{itemize}
	\item $\alpha$, then the functions are continuous for $x\in\{0, 1\}$;
	\item $x$, then the functions are continuous for $\ell < \infty$ and $\ve$ small enough;
	\item $\ve$, then the functions are continuous for $x\notin\eper(f_{\alpha})$. 
\end{itemize}
Thus $\mathrm{c}(\cdot, 0, \cdot, \ell)$ and $\mathrm{det}(\cdot, 0, \cdot, \ell)$ are continuous functions for every $l\leq\infty$ on the set $(0, 1/2)\times (0, 1]$.

By the previous theorem, it is not interesting to study functions $\utdet_\ell(\alpha) = \lim\inf_{\ve\to 0^{+}}\tdet{\ell}{f_{\alpha}}(\ve)$ and $\otdet_\ell(\alpha) = \lim\sup_{\ve\to 0^{+}}\tdet{\ell}{f_{\alpha}}(\ve)$ for $\ell < \infty$. This corresponds to the original hypothesis about measures of predictability. We have already seen that the hypothesis is not true for $\ell = \infty$. The remainder of this paper will be devoted to the study of functions $\utdet$ and $\otdet$, that is we will show how wrong the hypothesis was.

\begin{lema}
\label{lema:liminf_conti}
Functions $\utdet, \otdet$ are continuous.
\end{lema}
\begin{proof}
From Theorem~\ref{thm:prechod} and Lemma~\ref{lema:int_conti}, the determinism is the uniformly continuous function of $\alpha$ and $\ve$ at $J_{\alpha}\times J_{\ve}$ where $J_{\ve}\subset (0, 1]$ and $J_{\alpha}\subset (0, 1/2)$ are compact intervals. Thus for $\varepsilon > 0$ there is $\delta > 0$ with the property that if $\diam(J_{\ve}) < \delta$ and if $\diam(J_{\alpha}) < \delta$, then for $(\alpha_1, \ve_1), (\alpha_2, \ve_2)\in J_{\alpha}\times J_{\ve}$ the difference of determinisms $|\tdet{\infty}{f_{\alpha_1}}(\ve_1) - \tdet{\infty}{f_{\alpha_2}}(\ve_2)|$ is less than $\varepsilon$.

Fix $\alpha$ and let $h\geq 0$ be such that $\alpha^h < 1/\alpha\cdot(1 - 2\alpha)$. Choose $0 < \delta_h < \delta/2$ to satisfy $(\alpha + \delta)^h < 1/(\alpha + \delta_h)\cdot(1 - 2(\alpha + \delta_h))$ and $(\alpha + \delta_h)^h - (\alpha - \delta_h)^{h + 1} < \delta$. Put $J_{\ve} = [(\alpha - \delta_h)^{h + 1}, (\alpha + \delta_h)^h]$ and $J_{\alpha} = [\alpha - \delta_h, \alpha + \delta_h]$. Let $\alpha_1, \alpha_2\in J_{\alpha}$ and $\ve_1 = \alpha^h_1 - \alpha^{h + 2}_1, \ve_2 = \alpha^h_2 - \alpha^{h + 2}_2\in J_{\ve}$. Then $|\ve_2 - \ve_1| < \delta$. Moreover, $\utdet(\alpha_1) = \tdet{\infty}{f_{\alpha_1}}(\ve_1)$ and $\utdet(\alpha_2) = \tdet{\infty}{f_{\alpha_2}}(\ve_2)$. From previous assumptions, $|\utdet(\alpha_1) - \utdet(\alpha_2)| < \varepsilon$. Similarly, $|\otdet(\alpha_1) - \otdet(\alpha_2)| < \varepsilon$.
\end{proof}

Let $\varepsilon > 0$. From Corollary~\ref{cor:rozdiel_ck_c}, choosing an even $k > 0$ such that for every $\alpha\in (0, 1/2)$, $h = k/2$ and $\ve \in (\alpha^{h + 1}, \alpha^h]$, we have
\begin{align*}
\left|\tdet{\infty}{f}(\ve) - \tdet{\infty}{f_k}(\ve)\right| \leq \frac{24}{2^{k - h - 1} - 8} = \frac{24}{2^{k/2 - 1} - 8} < \frac{\varepsilon}{2}.
\end{align*}
By an easy computation, $2^{-k/2 + 6} < \varepsilon$. Let $\alpha$ be such that $1 - 2\alpha < \alpha^k$. 

Let $\tdet{\infty}{f_k}(\ve) < d + \varepsilon/2$, then
\begin{align*}
\left|\tdet{\infty}{f}(\ve) - d\right|\leq \left|\tdet{\infty}{f}(\ve) - \tdet{\infty}{f_k}(\ve)\right| + \left|\tdet{\infty}{f_k}(\ve) - d\right| < \frac{\varepsilon}{2} + \frac{\varepsilon}{2} = \varepsilon.
\end{align*}
Hence, for the purpose of the next corollary it is sufficient to show that
\begin{itemize}
	\item $\tdet{\infty}{f_k}(\ve) < 1/2 + \varepsilon/2$ for all small enough $\ve$ and
	\item $\tdet{\infty}{f_k}(\ve) < 1/3 + \varepsilon/2$ for some $\ve > 0$.
\end{itemize}
For $\alpha$ considered here, we thus get $\otdet(\alpha) < 1/2 + \varepsilon$ and $\utdet(\alpha) < 1/3 + \varepsilon$.

Using~(\ref{eq:gap_seq_dis}), for every $u\in\Sigma^k$ where $u \leq 0 1^{k - 1}$, we get
\begin{align*}
\dis(0 1^{k - 1}, 0 1^{k - 1} \pp (\gamma(u) - 2)) \leq \dis(0^k, 0^k \pp (\gamma(u) - 2)) + \alpha^k < \dis(0^k, 0^k \pp (\gamma(u) - 1)).
\end{align*}
As $\alpha^{h - 1} < \alpha^{k - 1}$, we have $\dis(0^h 1^{k - h}, 0^h 1^{k - h} \pp \gamma(u))\geq \dis(0^k, 0^k \pp \gamma(u))$ for every $u\leq 0^h 1^{k - h}$.

Since $A(2^{m_1}, j) \leq A(2^{m_2}, j)$ where $m_1\leq m_2$ and $j \leq 2^{m_1}$, it follows that
\begin{align*}
\dis(0 1^{k - 1}, 0 1^{k - 1}\pp \gamma(u))	&\geq \dis(0^2 1^{k - 2}, 0^2 1^{k - 2} \pp \gamma(u)) \geq \ldots \geq \\
																						&\geq \dis(0^h 1^{k - h}, 0^h 1^{k - h} \pp \gamma(u))\geq\dis(0^k, 0^k\pp\gamma(u))
\end{align*}
for every $u \leq 0^h 1^{k - h}$. Put $j = \gamma(u)$. If $\dis(0^k, 0^k \pp (\gamma(u) - 1))\leq \ve$ and if $\dis(0^k, 0^k \pp \gamma(u)) > \ve$, then from patterns $A_0$, $A_1$, $B_1$, $C_1$ and Corollary~\ref{dosledok:pocet_blizke_orbity},
\begin{align*}
\tc{1}{f_k}(\ve)\geq \frac{1}{2^{2k}}&\cdot\left(2^{k + 1}j + 3j - j^2 - 3\cdot 2^k - 2\right) \geq 2^{- h - 1}\quad\text{and}\quad
\tc{\infty}{f_k}(\ve)\leq \frac{2^k j}{2^{2k}}.
\end{align*}
Now, we want to show that $\tdet{\infty}{f_k}(\ve) < 1/2 + \varepsilon/2$ for every $\ve\in (\alpha^{h + 1}, \alpha^h]$ (and by Lemma~\ref{lema:det_epsilon}, for all $\ve\leq\alpha^h$). From previous inequalities, we thus want to show that
\begin{align*}
\tdet{\infty}{f_k}(\ve)\leq\frac{2^k j}{2^{k + 1}j + 3j - j^2 - 3\cdot 2^k - 2} & < \frac{1}{2} + \frac{\varepsilon}{2}\\
j^2 + 3\cdot 2^k + 2 & < 3j + \varepsilon(2^{k + 1}j + 3j - j^2 - 3\cdot 2^k - 2).
\end{align*}
From assumptions, $\varepsilon > 2^{-k/2 + 6}$, $j \geq 2^{k/2 - 1}$, $2^{k + 1}j + 3j - j^2 - 3\cdot 2^k - 2 \geq 2^{3k/2 - 1}$, and $j^2 \leq 2^k$. Thus
\begin{align*}
3j + \varepsilon(2^{k + 1}j + 3j - j^2 - 3\cdot 2^k - 2) > 3\cdot 2^{k/2 - 1} + 2^{k + 5} > 2^{k + 2} + 2 > j^2 + 3\cdot 2^k + 2.
\end{align*}
\newline

We now show that $\utdet$ not only goes below $1$, but also can be arbitrarily close to $1/3$. Let $\ve = \alpha^h - \alpha^{h + 2} - \alpha^{h + 1}\cdot (1 - 2\alpha)$. Then,
\begin{align*}
&D_k(\ve)[1, 2^{k - h} + 2^{k - h - 2}] = D_k(\ve)[2^{k - 1}, 2^{k - 1} + 2^{k - h} - 2^{k - h - 2} - 1] = 1 \qquad and \\
&D_k(\ve)[1, 2^{k - h} + 2^{k - h - 2} + 1] = D_k(\ve)[2^{k - 1}, 2^{k - 1} + 2^{k - h} - 2^{k - h - 2}] = 0.
\end{align*}
Using patterns $A_1$, $B_1$, $C_1$, the correlation integral $\tc{1}{f_k}(\ve)$ is greater than or equal to $1/2^{2k}\cdot\left[10\cdot 2^{3/2k - 2} - 41\cdot 2^{k - 4} + 5\cdot 2^{k/2 - 2}\right]$. The correlation integral in the infinite horizon $\tc{\infty}{f_k}(\ve)$ is equal to $1/2^{2k}\cdot\left[2^{3/2k - 1}\right]$ (Corollary~\ref{dosledok:pocet_blizke_orbity}). We now show that
\begin{align*}
\tdet{\infty}{f_k}(\ve) \leq \frac{8\cdot 2^{3/2k}}{40\cdot 2^{3/2k} - 41\cdot 2^k + 20\cdot 2^{k/2}} < \frac{1}{3} + \frac{\varepsilon}{2}.
\end{align*}
By an easy computation, we have to prove that
\begin{align*}
2^k\cdot(41 + 123/2\varepsilon) < 16\cdot 2^{3/2k} + 20\cdot 2^{k/2} + 60\cdot 2^{3/2k}\varepsilon + 15/2\cdot2^{k/2}\varepsilon.
\end{align*}
Using $2^{-k/2 + 6} < \varepsilon \leq 2/3$, $82 < 2^7$, and $60 > 2^5$, the statement follows.

\begin{dosledok}
\label{cor:liminf}
For all $\overline{d}\in (1/2, 1], \underline{d}\in (1/3, 8/15]$ there are $\overline{\alpha}, \underline{\alpha}\in (0, 1/2)$ satisfying $\otdet(\overline{\alpha}) = \overline{d}$ and $\utdet(\underline{\alpha}) = \underline{d}$.
\end{dosledok}


\section*{Acknowledgements}
The author would like to thank Vladim\'ir \v Spitalsk\'y and Mari\'an Grend\'ar for their helpful suggestions
and comments during discussions.

The work was supported by Matej Bel University, the Slovak Grant Agency under the grant number VEGA~1/0786/15. Also, the author gratefully acknowledges support from Slovanet, a.s..


\end{document}